\documentclass{article}
\usepackage[a4paper]{geometry}
\usepackage[utf8]{inputenc}

\usepackage[colorlinks]{hyperref}
\usepackage[singlelinecheck=false]{caption}
\usepackage{booktabs}

\usepackage{graphicx}
\usepackage{latexsym}

\usepackage{lipsum}
\usepackage{graphicx}
\usepackage{authblk}
\usepackage[english]{babel}
\usepackage{color}
\usepackage{hyperref}
\hypersetup{
    colorlinks=true,
    linkcolor=blue,
    filecolor=blue,      
    urlcolor=blue,
    citecolor=blue,
    }
\usepackage{dsfont}
\usepackage{geometry}
\usepackage{float}
\usepackage{bbm}
\usepackage{enumerate}

\usepackage[nottoc]{tocbibind}
 \usepackage{indentfirst}

\graphicspath{{./figure/}}

\usepackage{pgfplots}
\usepackage{comment}
\usepackage{fontenc}
\usepackage{booktabs}
\usepackage{tabularx}
\usepackage{tikz}
\usepackage{caption}
\usepackage{subcaption}
\usepackage{comment}

\DeclareCaptionLabelFormat{custom}
{%
      \textbf{Figure #2}
}
\DeclareCaptionLabelSeparator{custom}{ }
\DeclareCaptionFormat{custom}
{%
\centering
    #1 #2 #3 
}
 \usepackage{tikz-cd}

\usepackage{amsmath,amsfonts,amssymb,amsthm}

\numberwithin{equation}{section}
\usepackage{etoolbox}
\apptocmd{\thebibliography}{}{}{}
\theoremstyle{plain}
\newtheorem{theorem}{Theorem}[section]
\newtheorem{lemma}[theorem]{Lemma}
\newtheorem{corollary}[theorem]{Corollary}
\newtheorem{proposition}[theorem]{Proposition}

\newtheorem{reminder*}{[theorem]Reminder}

\newtheorem{details*}[theorem]{Details}

\newtheorem{comm*}{Comment}
\newtheorem{example}[theorem]{Example}
\newtheorem{definition}[theorem]{Definition} 
\newtheorem{definition*}{[theorem]Definition}

\newtheorem{notation*}{Notation}

\newtheorem{remark}[theorem]{Remark}

\title{The detailed balance property and chemical systems out of equilibrium}
  \author{E. Franco, J. J. L. Velázquez}

\begin{document}

\maketitle

\begin{abstract}
The detailed balance property is a fundamental property that must be satisfied in all the macroscopic systems with a well defined temperature at each point.
On the other hand, many biochemical networks work in non-equilibrium conditions and they can be effectively modelled using sets of equations in which the detailed balance condition fails. 
In this paper we study a class of "out of equilibrium" chemical networks that can be obtained freezing the concentration of some substances in chemical networks for which the detailed balance property holds. 
In particular, we prove that any chemical system with bidirectional chemical reactions can be extended to a system having additional substances and for which the detailed balance property holds. 
\end{abstract}


\tableofcontents
\section{Introduction}
A property that must be satisfied at the fundamental level by any chemical system with well defined thermodynamic quantities at each macroscopic point, as the temperature and the concentration of substances, is that the chemical rates should satisfy the so-called detailed balance property. 
This property, that was introduced for the analysis of collisions in gases by Boltzmann (see \cite{boltzmann1964lectures}), states that at equilibrium each reaction is balanced by its reverse reaction. 
However, many biochemical networks modelling biological processes contain irreversible reactions and therefore they do not satisfy the detailed balance condition. 
One way of justifying these models is to assume that they describe an open system, which is in contact with one or more reservoirs of substances. As a consequence these biochemical networks operate in "out of equilibrium" conditions. 

In this paper we will be concerned only with chemical systems in which the temperature is constant, therefore the systems considered in this paper are assumed to  be in contact with a reservoir at constant temperature. The chemical systems that we consider exchange heat with the environment.  
Moreover, the chemical networks that we study in this paper are endowed with mass action kinetics, hence they are chemical networks with reaction rates given by the mass action law. We stress that we make this choice partly because the mass action assumption simplifies the analysis of the kinetic systems and partly because, many non-mass action kinetics can be obtained as limits of mass action kinetics. This is the case for instance for the Michaelis-Menten kinetics or for the Hill law (see for instance \cite{goldbeter1981amplified,segel1989quasi}). 

The kinetic systems that we consider in this paper are "out of equilibrium" due to the exchange of matter with the environment. 
In particular, we assume that these kinetic systems are in contact with substances whose concentration is out of equilibrium. 
Under some assumptions that will be prescribed in this paper, these systems can be described by effective models in which the detailed balance property fails.
From this point of view, the failure of detailed balance provides a measure for the lack of equilibrium of the system. 

One of the issues that we address in this paper is how to derive kinetic systems that are out of equilibrium taking as a starting point kinetic systems that satisfy the detailed balance condition. 
In other words, we assume that the detailed balance property should be satisfied at the fundamental level by any chemical system and the lack of detailed balance takes place only as an effective property that arises from the fact that the concentrations of some substances in the network are kept at "non equilibrium" values in an active manner. 
The reason why these concentrations are kept out of equilibrium could be, for instance, an exchange of chemicals between the systems under consideration (for instance a subset of chemicals inside a cell) and the environment which is assumed to be out of equilibrium. Or alternatively the concentration of a substance could be kept at non-equilibrium concentrations by means of an active mechanism, for instance the production of ATP in mitochondria. 

Let us mention that kinetic systems that exchange chemicals with the environment can be found in the mathematical and physical literature. We can refer for instance to \cite{craciun2005multiple,craciun2006multiple,craciun2010multiple,schnakenberg1976network}. 
One of the questions that we examine in this paper is the following. 
Suppose that we have a kinetic system that satisfies the detailed balance property. Suppose that some concentrations are frozen at constant values.
Then the non-frozen concentrations solve the equations associated with a kinetic system that we will denote as \textit{reduced kinetic system}. 
As a matter of fact, it turns out that one of the ways in which it is possible to obtain a reduced kinetic system satisfying the detailed balance condition, is assuming that the frozen concentrations are at equilibrium values. 
Alternative ways to obtain a reduced kinetic system that satisfies the detailed balance property are also discussed in Section \ref{sec:detailed balance reduced}. 
In particular, we prove that the reduced system satisfies the detailed balance property in a robust way (i.e. the property is stable under small changes of the reaction rates and of the values of frozen concentrations) if and only if some topological condition on the reactions belonging to the cycles are satisfied both by the reduced and the original system. 
We remark that the fact that some topological conditions must be imposed on the cycles of the kinetic systems in order to obtain the detailed balance property in a robust manner is not surprising, as it is well known that the detailed balance property imposes conditions on the rates of the reactions that are part of a cycle (see for instance the \textit{circuit condition} or \textit{Wegscheider criterion} in \cite{wegscheider1902simultane}).

We stress that freezing the concentrations of certain chemicals at constant values is not the only way in which we can obtain non-equilibrium effects. We can imagine an exchange of matter taking place at the same time scale as the evolution inside the kinetic system, for instance we might have that the concentration  $n_F $ of a certain substance $F$ is given by 
\begin{equation} \label{eq:intro const fluxes}
\frac{d}{ dt } n_F =J_F + \alpha (n_{ext}- n_F ).
\end{equation}
Here $\alpha >0 $ and $J_F$ represent the changes in the concentration of $n_F $ due to the dynamics taking place in the kinetic system under consideration. Instead, $\alpha n_{ext}$ represents the influx of substance coming from the environment, while the term $-\alpha  n_F$ accounts for the outflux of substance from the kinetic system to the environment. 
In that case the concentration of the substance $F$ is not frozen, but it changes in time. 
Examples of models in which the concentration of a substance or of many substances change in time due to outfluxes and influxes of matter are for instance models in which there is an active or passive transfer of chemicals through a membrane.

As indicated above there is a plethora of models describing biochemical systems by means of ODEs for which the detailed balance condition is not satisfied. We can refer for instance to some models in \cite{alon2019introduction} or the models of adaptation in \cite{barkai1997robustness,ferrell2016perfect}, the model of the Calvin cycle (see for instance \cite{gurbuz2021analysis,rendall2014dynamical}), the kinetic proofreading model (see \cite{franco2024stochastic,hopfield1974kinetic}), the model of ABC transporters \cite{flatt2023abc}, the models of adaptation in \cite{ferrell2016perfect} and some of the  models in \cite{milo2015cell} (for instance the models of chemotaxis) and in \cite{phillips2012physical}. 
Since we stated above that at the fundamental level every biochemical model should satisfy the detailed balance condition, it is natural to ask under which conditions a kinetic system can be obtained by means of the reduction of a kinetic system for which the detailed balance condition holds. 

In this paper we will be mostly concerned with the study of systems in which all the reactions are bidirectional, i.e. they have the form 
\[ 
(1) +(2) \leftrightarrows (3)+(4). 
\]
One directional reactions have, instead the form 
\[
(1) +(2) \rightarrow (3)+(4). 
\]
The models in which the detailed balance property holds are necessarily bidirectional.
Conditions in one-directional networks that can be obtained as a limit of networks for which the detailed balance property holds have been found in \cite{gorban2011extended}. 

We prove that any bidirectional biochemical system can be obtained as the reduction of a larger kinetic system for which the detailed balance property holds, we refer to this system as the \textit{completed kinetic system}. 
More precisely, we prove that every open kinetic system can be obtained as the reduction of a \textit{closed} kinetic system. 
In this paper we say that a kinetic system is closed if it does not exchange substances with the environment. 
Since, as explained before, the lack of detailed balance emerges as an effective property of kinetic systems in which there are influxes and outfluxes of matter, closed kinetic systems are kinetic system that satisfy the detailed balance property and that do not exchange substances with the environment. 
The type of completion studied here has some similarities with the completion of a model of the Calvin cycle explained in \cite{rendall2017calvin}. In that paper, indeed, the relation between a model of Calvin cycle where the concentration of ATP molecules is assumed to change in time due to the reactions taking place in the network and a model in which the concentration of ATP molecules is chosen at constant values is studied.
The concentration of ATP is assumed to be at constant values because the concentration of ATP is larger than the concentration of the other substances in the Calvin cycle. 
Here we focus on understanding whether or not the property of detailed balance of the completed system is inherited by the reduced system.

On the other hand, for some specific kinetic systems there might be constraints in the way in which reactions can be modified in order to obtain a closed completion, for instance because some chemical reactions are known in full detail. 
We prove that, when this is the case, it is not always possible to find a closed completion for these kinetic systems. 
Indeed, as we will explain later in detail, the fact that the kinetic system can be completed or not depends on the position of the reactions that cannot be modified in the network, in particular the difficulties arise when these reactions belong to cycles. 

One of the main reasons to study the relation between the systems satisfying the detailed balance property and more general chemical networks is that this approach allows to measure the amount of "non-equilibrium" of these chemical networks by means of the fluxes of matter and energy required to make the system functioning. 
One of the issues in which we are interested is determining the degree of lack of equilibrium required for a chemical network to be able to perform some biological function. 
One example of this type of function is the so-called adaptation property (i.e. the capability of a network to react to changes in the signal, instead of reacting to the absolute value of it. This property is exhibited by many biological sensory systems ranging from bacteria to multicellular organisms. 
In \cite{franco2025adaptation} we will prove that the adaptation property cannot take place for closed chemical systems in a robust manner. The relation between the adaptation property and the detailed balance property will be studied in detail in \cite{franco2025adaptation}. 

\bigskip 

\textbf{Plan of the paper and main results} 

\bigskip 

In Section \ref{sec:review} we will review some of the basic definitions and results on kinetic systems and on chemical networks. 
 In Section \ref{sec:chemical reaction networks} we will remember some of the basics concepts used in the theory of chemical networks, as for instance the notion of stoichiometric subspace, and of conservation laws.
In Section \ref{sec:kinetic system} we will review the definition of kinetic systems (i.e. chemical networks associated with a rate function).

In Section \ref{sec:kinetic system DB} we review some of the main results for kinetic systems that satisfy the detailed balance property.
The reason why we revise the properties of kinetic systems that satisfy the detailed balance condition is that, as mentioned above, this property is a fundamental property of kinetic systems describing reactions in systems with well defined thermodynamic quantities (in local thermal equilibrium). 
In Section \ref{sec:DB} and Section \ref{sec:energy} we recall the definition of detailed balance and we recall that this property allows to associate to every substance $i \in \Omega $ in the kinetic system an energy $E_i$.
In the chemical literature there is usually a well defined vector of energies $E \in \mathbb R^N $. On the other hand, the equilibrium solution $N=(N_k) $ of a chemical system for which the detailed balance property holds is of the form 
 $N=(N_k)_k=e^{- E_k \pm \sum_{j=1}^L \mu_j m_j(k) } $ where $m_j $ for $j =1, \dots L $ are the conserved quantities and $\mu_j$ are the so-called chemical potentials. It turns out that the vector of energies $E$ can be defined up to the addition of the quantity $\sum_{j=1}^L \mu_j m_j(k)$. Due to this we will talk in the following of a whole class of vectors of energies, which can be defined up to the addition of the quantity $\sum_{j=1}^L \mu_j m_j(k)$. In the cases in which we need to use different vectors of Gibbs free energies $E$ and $\tilde E$ for a given chemical network, we will need to take into account explicitly the role of the chemical potentials.

 Finally, we recall that it is possible to associate to kinetic systems satisfying the detailed balance property a non-increasing free energy. Among the class of concentrations in the same stoichiometric class, the energy is minimized at the steady states. This means that when the kinetic system is at the steady state, then it does not dissipate energy. 
Moreover, the non-increasing free energy serves as a Lyapunov functional that allows to study the long-time behaviour of kinetic systems with the detailed balance property.
Finally, we introduce the definition of closed kinetic systems. These are kinetic systems that satisfy the detailed balance property and do not exchange substances with the environment. 

In Section \ref{sec:reduction} we define the reduced kinetic systems and study some of their properties. Reduced kinetic systems are effective systems used in order to describe the dynamics of systems in which certain chemicals are kept at constant concentration due to exchanges of substances with the environment. As mentioned above these reduced kinetic systems do not necessarily satisfy the detailed balance property. 

We start by explaining how we define the reduction of the reactions of a chemical network. This is done in Section \ref{sec:reduction of chem netowrks}.
  In Section \ref{sec:reduced cycles} we study the relation between the cycles of the reduced chemical network and the cycles of the non-reduced one. Notice that we are interested in the cycles of the kinetic systems because the detailed balance property is satisfied by a kinetic system if and only if the so-called circuit condition is satisfied along the cycles (\cite{wegscheider1902simultane}). As a consequence, if a chemical system does not have cycles and also its reduction does not have cycles, then the associated kinetic systems will satisfy the detailed balance property for every choice of reaction rates. In Section \ref{sec:reduction of kinetic systems} we introduce the definition of reduced kinetic system, these are reduced chemical system endowed with reaction rates that depend on the concentration of the frozen substances and on the reaction rates of the non-reduced kinetic system. 

  In Section \ref{sec:DB in the reduction} we study under which conditions the reduction of a kinetic system  that satisfies the detailed balance property also  satisfies the detailed balance condition and under which conditions it does not. 
  The main results in Section \ref{sec:detailed balance reduced equilibrium} are summarized in the following informal theorem that will be stated later precisely (Proposition \ref{prop: DB when equilibirum conc}). 
  \begin{theorem}
  Consider a kinetic system that satisfies the detailed balance property. Assume that the concentrations of some substances are kept at constant values by influxes and outfluxes of chemicals. Moreover, assume that if the substances whose concentrations are frozen appear in the cycles, then their concentrations are at equilibrium values. Then the corresponding reduced kinetic system satisfies the detailed balance property.    
  \end{theorem}
    Here with equilibrium values we mean that the concentrations of the frozen substances are given by $e^{- E} $, where $E$ is a vector of energies associated to the non-reduced kinetic system. 
    
    The main result of Section \ref{sec:detailed balance reduced} is summarized in the following informal theorem. (This theorem summarizes the results in Theorem \ref{thm:same cycles implies db}, Proposition \ref{prop: instability of (DB) different cycles}, Proposition \ref{prop: instability of (DB) no one to one}, Proposition \ref{prop: instability of (DB) one to one and zeros}). 
    \begin{theorem}
    Consider a kinetic system that satisfies the detailed balance property. Then the reduced kinetic system obtained by freezing some concentrations does not satisfy the detailed balance property unless either the reaction rates are fine tuned, or the frozen concentrations are chosen at equilibrium values, or the cycles of the reduced system and of the non reduced system are the same.
    \end{theorem}

    In Section \ref{sec:completion} instead we analyse if a kinetic system that is not closed admits a completion that is closed. 
    The main result of this section is Theorem \ref{thm:completion}, which informally written states the following. 
       \begin{theorem}
    Consider a kinetic system  that is not closed. Then it can be obtained as the reduction of a closed kinetic system.
    \end{theorem}
    This theorem guarantees the admissibility of any bidirectional chemical network. Here an admissible network is a network that is not closed, but can be obtained as the reduction of a closed kinetic system. 
    
In Section \ref{sec:kinetic system with fluxes} we define kinetic systems with fluxes. These are kinetic systems in which we have influxes and outfluxes of chemicals.
Let us stress that in this paper we study kinetic systems with fluxes that keep the concentration of some substances constant in time. As a consequence, if we focus on the substances that change in time, the dynamics of kinetic systems with fluxes and of the reduced kinetic systems is the same and therefore kinetic systems with fluxes are just reformulations of reduced kinetic systems.
The advantage of working with kinetic systems with fluxes instead of with reduced kinetic systems is that it is possible to obtain that the free energy of the kinetic system with fluxes $F$ satisfies the following equality
\[
\partial_t F = -\mathcal D_R + J^{ext}. 
\]
 Here $\mathcal D_R $ is the dissipation of free energy of the reduced kinetic system, while $J^{ext} $ is the contribution to the free energy due to the external fluxes.

\subsection{Notation}
 We define $\mathbb R_+$ and $\mathbb R_*$ to be given respectively by $ \mathbb R_* = [0, \infty )$ and $ \mathbb R_+ = (0, \infty )$. 
In some cases, to help the reader we indicate with 
$\textbf{0}_d $ the zero vectors belonging to $\mathbb R^d $. 
Moreover we denote with $e_i \subset \mathbb R^n $, for $ i \in 1, \dots n $, the vectors of the canonical basis of $\mathbb R^n$. 
Given two vectors $v_1 , v_2 \in \mathbb R^n $ we denote with $\langle v_1, v_2 \rangle $ their euclidean scalar product in $\mathbb R^n$.
Moreover, given a vector $v \in \mathbb R^n $, we will denote with $e^{v}$ the vector $ (e^{v(i)})_{i=1}^n \in \mathbb R^n$. Similarly it will be useful to denote with $\log(v) $ the vector $ (\log(v(i)))_{i=1}^n \in \mathbb R^n$. 

\section{Kinetic systems} \label{sec:review}
 In this section we recall the definition of kinetic systems and of some of their properties.
 This allows us to fix the notation that will be used later and to recall some properties that will be repeatedly used in the paper. 
 
 A kinetic system is a set of substances that interact via some chemical reactions, that take place at some given rates. 
 In some cases, the topological properties of a kinetic system, determine the qualitative behaviour of the kinetic system, which turns out to be independent on the rates of the reactions. An example of a qualitative property, that in some cases depends only on the topological properties of the network is the property of detailed balance that holds for every linear chemical network that is a tree (see \cite{feinberg2019foundations}). 
 This is the reason why we start this section by introducing the concept of chemical network, which is a set of substances and a set of chemical reactions and later we will define kinetic systems as a chemical networks to which we associate a kinetics, i.e. to which we associate the rates of the reactions.

\subsection{Chemical reaction networks} \label{sec:chemical reaction networks}
In this section we give the definition of chemical networks and introduce some of their properties, for instance we provide the definition of conservative chemical networks and of bidirectional chemical networks.
As we will see, a chemical network consists of a set of substances and a set of reactions and can be associated with a graph with vertices in $\mathbb Z^N$, where $N$ is the number of substances in the network. 
The properties that we formulate for chemical reaction networks are independent on the reactions rates associated to the reactions in $\mathcal R$. 
\begin{definition}[Chemical network]
  Let $\Omega:=\{ 1, \dots, N \}$. 
  Let $r \geq 1 $ and $\mathcal R := \{ R_1, \dots, R_r  \}  $ where $R_j \in \mathbb Z^N\setminus \{ 0\} $ for every $j \in \{ 1, \dots, r \} $. 
  Then we say that $(\Omega , \mathcal R )$ is a chemical network. 
\end{definition}
The elements of $\Omega$ are the \textit{substances} of the network, while $\mathcal R$ is the set of the \textit{reactions} of the network. Notice that we assume that $R \in \mathbb Z^N $. This means that the reactions that we consider take place between compositions that have a natural number of molecules of each substance.

Let $R \in \mathcal R $ be a reaction. We define the following three sets 
\[
I(R):=\{ i \in \Omega : R(i) <0 \}, \quad   F(R):= \{  i \in \Omega : R(i) >0 \} \ \text{ and } \ D(R):=I(R) \cup F(R).
\]
Hence, the chemical reaction $R$ transforms the chemicals in the set $I(R)$ into the chemicals in the set $F(R)$. The set $I(R)$ is the set of the initial substances, while $F(R) $ is the set of the final substances. 
In this paper we assume that $I(R) \cap F(R) = \emptyset $ for every $R \in \mathcal R$. Another assumption that we make in this paper is that every substance in $i \in \Omega $ belongs to at least one reaction, i.e. for every $i \in \Omega $ there exists an $R \in \mathcal R$ such that $ i \in  D(R)$ and that $D(R)\neq \emptyset$ for every $R \in \mathcal R$ (otherwise we would just remove the substance $i $ from the set of substances $\Omega$).

The set of the integer compositions that are accessible starting from the composition $\textbf{0}_N$ is the set $\mathcal A(0)$ defined as 
\begin{equation}\label{A(0)}
\mathcal A(0)= \{ v \in \mathbb Z^N : v= \sum_{k=1}^r \lambda_k R_k \text{ for some }  \{ \lambda_k \}_{k=1}^r \subset \mathbb Z.   \} 
\end{equation}
Similarly, we can define also the set  of the states accessible from a state $v_0 \in \mathbb Z^N $, this is 
\[ 
\mathcal A (v_0):= v_0 + \mathcal A(0).
\]
A chemical network $(\Omega, \mathcal R )$ can be viewed as a graph $ \mathcal G = (E, V) $ with a countable number of vertices in $\mathbb Z^N $. 
Indeed, we can define the graph $\mathcal G =(V, E) $ where
$V:=\mathcal A(0)$ and 
\[
E:= \{ (v_1, v_2) \in V \times V  : v_1 -v_2 \in \mathcal R \}. 
\]
We stress that this graph is directed. 
We say that a chemical network is \textit{bidirectional} if for every $R \in \mathcal R $ we have that $- R \in \mathcal R$. Notice that this means that the graph $(V,E) $ induced by the chemical network is symmetric.

Given a network $(\Omega, \mathcal R)$ we consider the set of non-reverse reactions $\mathcal R_s \subset \mathcal R $, obtained identifying each reaction $R$ with the reversed reaction $- R$. 
More precisely, the set $\mathcal R_s \subset \mathcal R$ is defined as 
\begin{equation} \label{non reversed set}
\mathcal R_s := \{ R \in \mathcal R : - R \notin \mathcal R  \} \cup  \{ R \in \mathcal R \setminus \{ R \in \mathcal R : - R \notin \mathcal R  \} : \min{I(R) } < \min {F(R)} \}.
\end{equation}
Hence we have if $R \in \mathcal R $ is non-reversible, i.e. $-R \notin \mathcal R $, then $R \in \mathcal R_s $. Instead if $ R ,  - R \in \mathcal R$ only one of the two reactions belong to $\mathcal R_s$. As a consequence for every $R_{1} , R_{2 } \in \mathcal R_s$ we have $R_{1} \neq - R_{2} $.
Notice that if the network $(\Omega, \mathcal R)$ is bidirectional, then $|\mathcal R_s| =r/2$.
We associate to the set of the reactions $\mathcal R$ the matrix $\textbf R \in \mathbb Z^{ N \times |\mathcal R_s|} $ defined as 
\begin{equation}\label{matrix reactions} 
\textbf R_{j k }  =  R_k(j) \text{ where } j \in \{1, \dots, N \}, \ k \in \{ 1, \dots,  |\mathcal R_s|\} \ \text{ and where } \ R_k  \in \mathcal R_s.
\end{equation}
In this paper we refer to the matrix $\textbf{R} $ as the \textit{matrix of the reactions} associated with the network $(\Omega, \mathcal R)$. 

We can now define the notion of \textit{cycles} of a chemical network $(\Omega, \mathcal R)$.
\begin{definition}[Cycles of a chemical network]
Let $(\Omega , \mathcal R) $ be a chemical network. Assume that $\textbf{R} \in \mathbb Z^{ N \times |\mathcal R_s|} $ is the matrix of the reactions. 
The space of the cycles of the chemical network $(\Omega, \mathcal R)$ is defined as
\[
\mathcal C := \ker (\textbf{R}).
\]
\end{definition}
In particular, we say that a chemical network has no cycles if $\ker (\textbf{R}) = \{ \textbf{0}_{|\mathcal R_s| } \} $. 
It is possible to interpret the definition of cycles in terms of the cycles of the graph $\mathcal G = (V,E) $ as follows. 
The fact that $c \in \ker (\textbf{R}) $ implies that there exists a sequence of reactions $\{ R_{i}\} _{\{i: c(i) \neq 0\}} \subset \mathcal R $ that, when applied to the composition $\textbf{0}_N$, produces the composition $\textbf{0}_N$, indeed by definition we have that 
\[
\sum_{i=1}^r c(i) R_{i}=\textbf{0}_N. 
\]
Hence this means that there exists a cycle that contains the composition $\textbf{0}_N $ in the graph $\mathcal G $ defined above. 
The values $\{ |c(i)|\}_{i=1}^{|\mathcal R_s|}$ are the numbers of times that the reaction $\textbf{R} e_i=R_i \in \mathcal R$ appears in the cycle.
In the following it is convenient to use the notation $R\in \mathcal C$ to indicate that there exists a $i\in \{1, \dots, |\mathcal R|/2 \} $ such that $R=\textbf{R} e_i $ and such that there exists a $c \in \mathcal C $ with $c(i)\neq 0 $. 
In order to clarify the notation introduced here we conclude this subsection with an example of chemical network that has cycles. 
\begin{example}
    Consider the chemical network $(\Omega, \mathcal R) $ associated with the  chemical reactions 
    \[ 
(1) \  \rightarrow  \  (2) , \quad      (2) +(2)  \rightarrow  \ (1)+(1)  . 
\]
Notice that in this example we have one directional reactions. 
The matrix of the reactions is 
\[
\textbf{R}:= 
\left( 
\begin{matrix}
&-1 &  2   \\
& 1 & -2  \\
\end{matrix} \right). 
\]
The space of the cycles is $\mathcal C := \operatorname{span} \{ (2,1)\} $. 
If we apply twice the reaction $R_1=\left( \begin{matrix}
-1   \\
 1   \\
\end{matrix} \right) $  and once the reaction $R_2= \left( \begin{matrix}
2   \\
 -2   \\
\end{matrix} \right)$ to the composition $(0,0)$ we obtain again the composition $(0,0)$, indeed
    \[ 
 (0,0)  \ \underset{R_1} \rightarrow  \ (-1, 1) \ \underset{R_1}\rightarrow \  (-2, 2) \ \underset{R_2}\rightarrow \ (0,0) . 
\]
\end{example}

\subsubsection{Conservative networks}
In this section we introduce the definition of conservative network. To this end we start by introducing the definition of stoichiometric subspace. 
 \begin{definition}[Stoichiometric subspace]
    The stoichiometric subspace $\mathcal S  $ of a reaction network $(\Omega , \mathcal R) $ is
    \begin{equation} \label{eq:stochio}
\mathcal S  := \operatorname{span}\{ R: R \in \mathcal R  \}. 
    \end{equation}
\end{definition}
Notice that $\mathcal S \subset \mathbb R^N$.

\begin{definition}[Stoichiometric compatibility classes]
    Let $(\Omega, \mathcal R)$ be a chemical network. Let $\mathcal S$ be the corresponding  stoichiometric subspace. 
    Two vectors $n, n' \in \mathbb R_+^N $ are stoichimetrically compatible if $n- n' \in \mathcal S$.
\end{definition}
 The stoichiometric compatibility is an equivalence relation. Given a vector $v\in \mathbb R_+^N$, the equivalence class $[v]_\mathcal S$ generated by $v$ is defined as  $[v]_\mathcal S:=\{ x \in \mathbb R_+^N : x-v \in \mathcal S \}  $. 
Consider an initial vector $n_0 \in \mathbb R_+^N$ of concentration of substances in a chemical network  $(\Omega, \mathcal R)$.
Let $n(t)$ be the evolution of the concentration $n_0 $ due to the chemical reactions taking place in the network, then it holds that $n(t)- n_0 \in \mathcal S $. 
Therefore, independently on the reaction rates, $n(t)$ will be stoichimetrically compatible with $n_0$ for every positive time $t>0$, i.e. $n(t) \in [n_0]_{\mathcal S} $.

We now explain that we can associate a set of conservation laws to a chemical network. 
\begin{definition}[Set of conservation laws]
    The set $\mathcal M $ of conservation laws of a chemical network $(\Omega , \mathcal R) $ is defined as
    \begin{equation} \label{eq:stochio}
\mathcal M  := \mathcal S^\perp.  
    \end{equation}
\end{definition}
Notice that, by definition, given a conservation law $m \in \mathcal M $ we have that
    \[
     m^T R = \textbf{0},\quad \forall R \in \mathcal R. 
    \]
    As a consequence the vector of the concentrations of substances at positive times, $n(t)$ obtained as the evolution of an initial vector of concentrations $n_0$ is such that 
\[
 m^T n_0 = m^T n(t), \quad \text{ for every } t >0 \text{ and for any } m \in \mathcal M. 
\]
This explains why we refer to $\mathcal M $ as the set of the conservation laws. 
Notice that by definition $\mathcal M \subset \mathbb R^N$, however, for physically relevant chemical networks we expect conservation laws to be non-negative. 
This motivates the following definition of the set $\mathcal M_+$ of non-negative conservation laws
\[
\mathcal M_+:=\mathcal M \cap \mathbb R_*^N .
\]

Now that we introduced the definition of conservation law we can write the definition of conservative network. 
\begin{definition}[Conservative chemical network]
We say that the network $(\Omega, \mathcal R )$ is conservative if
\begin{equation}\label{eq:conservative}
\mathcal M_+ \cap \mathbb R_+^N \neq \emptyset.
\end{equation}
\end{definition}
In particular, a chemical network is conservative if and only if every substance $i \in \Omega $ appears in a conservation law as explained in the following lemma. 
\begin{lemma}
    Let $(\Omega, \mathcal R) $ be a chemical network. Then the following statements are equivalent. 
    \begin{enumerate}
        \item  The chemical network is conservative.  
        \item For every $j \in \Omega$ there exists a $m \in \mathcal M $ such that $m(j) >0$. 
    \end{enumerate}
\end{lemma}
\begin{proof}
    Assume that $2.$ holds. Then for every $j \in \Omega $ there exists a $m_j \in \mathcal M $ such that $m_j(j) >0$. Therefore $m = \sum_{j \in \Omega } m_j(j) \in \mathcal M $ is such that $m \in \mathbb R_+^N $ and therefore the network is conservative. The vice versa follows immediately by the fact that, by definition, there exists a $m \in \mathcal M_+ \cap \mathbb R_+^N $. 
\end{proof}

\begin{lemma} \label{lem:positivity of the conservation laws}
Assume that the chemical network $(\Omega, \mathcal R) $ is conservative. 
Then the set of the extreme rays of the positive cone $\mathcal M_+$ are a basis of $\mathcal M $. 
\end{lemma}
\begin{proof}
  Since the system is conservative there exists a vector $\overline m \in \mathcal M_+ \cap \mathbb R_+^N $. Let $V$ be the vector subspace of $\mathcal M $ generated by $\overline m$.  
  We define the following affine subspace $\mathcal L $ of $\mathcal M $ as $ \mathcal L := \overline m + V^\perp$.
Let us define the set 
  $\mathcal D:= \mathcal L \cap \mathcal M_+ $. The set $\mathcal D$ is convex and compact by definition. Therefore Krein–Milman theorem (see \cite{conex1970An}) guarantees that the set $\mathcal D$ is the convex hull of its extremal points and the set of the extremal points of $\mathcal D$ is non-empty. Notice that the extremal points of $\mathcal D$ define the set of the extreme generators of the cone $\mathcal M_+$.
  The extreme generators are linear independent vectors, hence the number of extremal points is finite and they generate $\mathcal M_+$. The desired conclusion follows. 
\end{proof}

\begin{example}
Consider the chemical network corresponding to the reactions 
\[
(1)+(2) \rightarrow (3)+(3), \quad (3) \rightarrow  (2). 
\]
Notice that in this example we have one directional reactions. 
Then $\Omega :=\{1, 2, 3 \} $ and the set of the reactions is 
\[
\mathcal R := \left\{ \left(  \begin{matrix} 
   -    1 \\
    -  1 \\
       2   
\end{matrix} \right), \left(  \begin{matrix} 
   0  \\
      1 \\
       - 1   
\end{matrix} \right) \right\}.
\]
Then $\mathcal M = \operatorname{span}\{ (1,1,1)^T \}$.  
Therefore the chemical network $(\Omega, \mathcal R)$ is conservative. 
Instead, consider the chemical network induced by 
\[
(1)+(2) \rightarrow (3), \quad (3) \rightarrow  (2). 
\]
It is easy to see that $\mathcal M = \operatorname{span}\{ (0,1,1) \}$.  
Therefore the chemical network is not conservative.
\end{example}

\subsection{Kinetic systems: chemical networks endowed with a rate function}
\label{sec:kinetic system}
In this section we state the definition of kinetic systems. A kinetic system is a chemical network to which we associate a kinetics. In particular in this paper we will always consider mass action kinetics, hence it is enough to associate to the set of reactions a set of reactions rates.

A \textit{reaction rate function} $\mathcal K : \mathcal R \rightarrow \mathbb R_+$ is a function that associate to each reaction its rate, i.e. $\mathcal K: R \mapsto \mathcal K(R)=:K_R \in \mathbb R_+$. 
We are now ready to give the definition of kinetic system. 
\begin{definition}[Kinetic system]
A kinetic system $(\Omega , \mathcal R, \mathcal K)$ is a chemical network $(\Omega , \mathcal R )$ endowed with the reaction rate function $\mathcal K$. The kinetic system $(\Omega, \mathcal R , \mathcal K )$ is bidirectional if the chemical network $(\Omega , \mathcal R )$ is bidirectional and is conservative if the chemical network $(\Omega , \mathcal R )$ is conservative. 
\end{definition}
We can associate to a kinetic system $(\Omega, \mathcal R, \mathcal K ) $ a system of ODEs describing the evolution in time of the concentrations of species in the network, $ n:=(n_1, \dots, n_N)^T \in \mathbb R_*^N $, i.e. 
\begin{equation} \label{ODEs}
\frac{d n (t) }{dt} = \sum_{ R \in \mathcal R } K_R R  \prod_{i\in I(R)} {(n_i)}^{-R(i)} , \quad n(0)=n_0 \in \mathbb R_*^N.
\end{equation}
We explain briefly how to interpret the system of ODEs. 
Let $n_i(t) $ be the concentration of the substance $ i \in \Omega $. The evolution of $n_i$ is driven by a loss and a gain term. The gain term is due to all the reactions that produce $i $, i.e. by all the reaction such that $ i \in F(R) $, hence $R(i)>0$. The gain term of each reaction is given by mass-action law, hence is given by $K_R \prod_{i \in I(R) }  (n_i)^{-R(i)}$. 
Similarly, the loss term is due to the reactions that are such that $ i \in I(R)$, hence $R(i) <0$. The contribution of the reaction $R $ to the loss term is given by $K_R \prod_{i \in I(R) }  (n_i)^{-R(i)}$.

The following lemma is useful to rewrite the system of ODEs \eqref{ODEs} in terms of the fluxes $J_R $ when the kinetic system is bidirectional. 
\begin{lemma}
    Assume that $(\Omega, \mathcal R, \mathcal K)$ is a bidirectional kinetic system. Then, for every $n \in \mathbb R_*^N $ we have that 
 \[
  \sum_{ R \in \mathcal R } K_R R  \prod_{i \in I(R)} {(n_i)}^{-R(i)} = \sum_{ R \in \mathcal R_s } R J_R(n), 
\]
where the fluxes $J_R (n) $ are defined as
\begin{equation} \label{eq:fluxes}
    J_R (n):= K_R \prod_{ j \in I(R)  } {(n_j)}^{-R(j)} - K_{- R }  \prod_{ j \in F(R)  } {n_j}^{R(j)}. 
\end{equation}
\end{lemma}
\begin{proof} 
Since $(\Omega, \mathcal R, \mathcal K)$ is bidirectional, then we have that 
\begin{align*}
  \sum_{ R \in \mathcal R } K_R R  \prod_{i \in I(R)} {(n_i)}^{-R(i)} &=  \sum_{ R \in \mathcal R_s } K_R R  \prod_{i \in I(R)} {(n_i)}^{-R(i)} - \sum_{ R \in \mathcal R_s } K_{-R} R  \prod_{i \in I(-R)} {n_i}^{R(i)} \\
  & = \sum_{ R \in \mathcal R_s } R \left(  K_R  \prod_{i \in I(R)} {(n_i)}^{-R(i)} - K_{-R}  \prod_{i \in F(R)} {n_i}^{R(i)} \right) \\
  &= \sum_{ R \in \mathcal R_s } R J_R(n). 
\end{align*}
\end{proof}
As a consequence, when the kinetic system $(\Omega, \mathcal R, \mathcal K)$ is bidirectional the system of ODEs \eqref{ODEs} can be written as 
\begin{equation} \label{ODEs fluxes}
\frac{dn(t)}{dt} =  \sum_{ R \in \mathcal R_s } R J_R(n), \quad n(0)= n_0 \in \mathbb R_*^N.
\end{equation}

\section{Detailed balance property of kinetic systems}\label{sec:kinetic system DB}
In this section we study one of the most important properties of kinetic systems: the property of detailed balance. 
As explained in the introduction, the detailed balance property is a fundamental property that follows by the physical principle of microscopic reversibility.  Every kinetic system that does not exchange substances with the environment satisfies the detailed balance property. However, as will be shown in Section \ref{sec:reduction}, the detailed balance property does not necessary hold for kinetic systems that exchange substances with the environment. 
From the mathematical point of view, an important consequence of the detailed balance property of kinetic systems is the existence of a unique positive stable steady state describing the long-time behaviour of the kinetic system.

In this Section we state results that are standard, but we collect them here in a way that is convenient for our purposes. 
The rest of this section will be organized as follows. In Section \ref{sec:DB} we introduce the definition of detailed balance. In Section \ref{sec:energy} we explain how we associate an energy to each composition in the network. In Section \ref{sec:global stability} we review the proof written in \cite{feinberg2019foundations} of the existence of a unique globally stable steady state for the ODEs system \eqref{ODEs} associated with a kinetic system that satisfies the detailed balance property. 

\subsection{Detailed balance property} \label{sec:DB}
In this section, we give the definition of detailed balance property for kinetic systems and we state three conditions that are equivalent to the detailed balance property. 
The detailed balance property for bidirectional kinetic system states that each reaction is balanced by its reverse reaction at the steady state. 

\begin{definition}[Detailed balance property]
  A bidirectional kinetic system $(\Omega, \mathcal R, \mathcal K ) $ satisfies the detailed balance property if there exists a $\overline N \in \mathbb R_+^N $ of \eqref{ODEs} such that  
  \begin{equation}\label{eq:DB}
K_R \prod_{i\in I(R)}  {(\overline N_i)}^{-R(i) } = K_{-R} \prod_{i \in F(R) } {\overline N_i}^{R(i) } \  \text{ for all } R\in \mathcal R_s. 
  \end{equation}
\end{definition}
Notice that by the definition of detailed balance we have that $\overline N$ is such that $J_R(\overline N) =0$ for every $R \in \mathcal R$. Therefore $\overline N$
is a steady state of \eqref{ODEs}.

We now state conditions that are equivalent to the detailed balance property and that are useful in order to check if a kinetic system satisfies the detailed balance property or not.
\begin{lemma} \label{lem:Db equivalence}
Let $(\Omega, \mathcal R , \mathcal K )$ be a bidirectional kinetic system. 
  The following statements are equivalent. 
  \begin{enumerate}
      \item The system satisfies the detailed balance property. 
      \item Condition \eqref{eq:DB} holds for every positive steady state of the system of ODEs \eqref{ODEs} corresponding to $(\Omega, \mathcal R, \mathcal K) $.
      \item Let $\textbf{R}$ be the matrix of the reactions.
      For every cycle $c \in \mathcal C $ it holds that 
      \begin{equation} \label{eq:cycles and DB}
      \prod_{ j=1}^{r/2} \left( \frac{K_{R_j}}{K_{-R_j}} \right)^{c(j)}  = 1, 
      \end{equation}
      here $r=|\mathcal R|$.
  \end{enumerate}
\end{lemma}
\begin{proof}
    For the proof of the equivalence between 1 and 2 we refer to the proof of Theorem 14.2.1 in \cite{feinberg2019foundations}. 
For the equivalence between property 3 and 1 we refer to \cite{feinberg1989necessary}.
\end{proof}
The condition 3. in Lemma \ref{lem:Db equivalence} is often referred in the literature as \textit{circuit condition} or \textit{Wegscheider criterion} (see \cite{wegscheider1902simultane}).
\subsection{Energy associated to kinetic systems with detailed balance} \label{sec:energy}
In this section we explain that when a kinetic system satisfies the detailed balance property, then it is possible to associate to each chemical in the system a vector of energies. 
As will be explained later this energy determines the values of the steady states of the system of ODEs \eqref{ODEs}.

We start by associating an energy to each of the reactions taking place in the chemical network. 
 Consider a kinetic system $(\Omega, \mathcal R)$ that is bidirectional. We define the \textit{Gibbs free energy function} associated to the set of reactions $\mathcal R $ as the function $\mathcal E: \mathcal R \rightarrow \mathbb R$ that maps each reaction to an energy as follows
\[
\mathcal E (R): = \log \left( \frac{K_{-R}}{ K_{R}} \right).  
\]
The energy $\mathcal E(R) $ associated with the reaction $R \in \mathcal R $ describe the change of energy that takes place during the reaction $R$. 

We extend the definition of the energy function $\mathcal E $ to sequences of reactions.
To this end we identify the vectors $x \in \mathbb R^{r/2} $ with the sequence $\{ ( x(i), R_i )\}_{i=1}^{r/2} $ where $R_i = \textbf{R} e_i$ and where the number $x(i) $ is the number of times that the reaction $R_i$ takes place. 
We define the energy function $\overline {\mathcal E} $ associated with a sequence of reactions as the function $\overline {\mathcal E }: \mathbb R^{r/2} \rightarrow \mathbb R $ defined by
\[
\overline{\mathcal E }(x) := \sum_{i=1}^{r/2} x(i) \mathcal E (R_i). 
\]
Notice that $\overline {\mathcal E }(x) $ is the sum of the energies of the reactions induced by the sequence $x$. Hence from the physical point of view, it is the total change in the energy of the system induced by the sequence of reactions $x$.
\begin{lemma} \label{lem:energy on cycles is zero}
Assume that $(\Omega, \mathcal R, \mathcal K ) $ is a kinetic system that satisfies the detailed balance condition.  
Let $c \in \mathcal C $, then $\overline{\mathcal E}(c) =0$. 
\end{lemma}
\begin{proof}
  The statement follows from Lemma \ref{lem:Db equivalence} and the definition of the map $\mathcal E$. Indeed 
         \[
\overline{\mathcal E}(c)= \sum_{j =1 }^{r/2} c(j) \mathcal E(R_j)  =0,
      \]
      where for every $j \in \{1, \dots, r/2\} $ we have that $R_j = \textbf{R} e_j $ and where $\textbf{R} $ is the matrix of the reactions. 
\end{proof}
Finally we extend the definition of energy to the mixtures/compositions that are formed in the network, i.e. to the set of the compositions $\mathcal A(0)$ reachable starting from the complex $\textbf{0}_N$ via the reactions in the chemical network. 
Assume that $\xi \in \mathcal A (0)$ where we recall that $\mathcal A(0) $ is defined by \eqref{A(0)}.
Then we have that there exists a $x \in \mathbb R^{r/2} $ such that $ \xi = \textbf{R} x $. We define the energy of the compositions in $\mathcal A(0)$ as the map $\tilde{\mathcal E} : \mathcal A(0) \rightarrow \mathbb R$ defined by 
\[
\tilde{\mathcal E} (\xi) := \overline{ \mathcal E}(x).
\]

Lemma \ref{lem:energy on cycles is zero} guarantees that $\tilde { \mathcal E} $ is well defined for kinetic systems that satisfy the detailed balance property as we explain in the following Lemma. 
\begin{lemma}
Assume that $(\Omega , \mathcal K , \mathcal R) $ is a kinetic system that satisfies the detailed balance property. 
    Then the map $\tilde{\mathcal E}$ is well defined. 
\end{lemma}
\begin{proof}
Assume that $\xi = \textbf{R} x_1 = \textbf{R} x_2 $. This implies that $x_1-x_2 \in \mathcal C$, hence $\mathcal E(x_1)- \mathcal E(x_2)=\mathcal E(x_1-x_2)=0$ and the fact that $ \tilde{\mathcal E} $  is well defined follows. 
\end{proof}
The energy $\tilde{\mathcal E}$ associated with the composition $\xi $ is just the change in the energy that is necessary to reach the composition $\xi $ starting from the composition $0$.
It is natural to expect that this energy is given by the sum of the energies of each of the chemicals in the composition $\xi $. 
In the following lemma we prove that this is the case. More precisely, we prove that if a kinetic system satisfies the detailed balance property, then the energy $\tilde{\mathcal E} $ associated with a mixture is additive, more precisely, it can be written as the weighted sum of the energies of the different substances that appear in the mixture, see \eqref{additivity}. 
\begin{lemma}[Additive energy]
    Assume that the kinetic system $(\Omega , \mathcal R , \mathcal K )$ satisfies the detailed balance property. 
    Let $w \in \mathbb R^{r/2} $ be defined as 
    \begin{equation} \label{w}
       w(i)=\mathcal E(R_i) \  \text{   for every } \  i=1, \dots , r/2 .
       \end{equation}
    Then there exists at least one solution $E \in \mathbb R^N$ to the equation 
    \begin{equation} \label{eq:energy}
    w = \textbf{R}^T  E.
    \end{equation}
    Moreover, for every $\xi \in \mathcal A(0)$ it holds that 
    \begin{equation} \label{additivity}
    \tilde{\mathcal E } (\xi ) = \sum_{j=1}^L \xi(j)   E (j)
    \end{equation}
    for every solution $E$ to \eqref{eq:energy}. 
\end{lemma}
\begin{proof}
 The detailed balance property guarantees that there exists a $ \overline N \in \mathbb R_+^N $ such that 
\[
\ln (K_R ) - \sum_{i \in I(R) } R(i) \ln ( \overline N_i ) = \ln (K_{-R} ) + \sum_{i \in F(R) } R(i) \ln ( \overline N_i)\quad  \forall R \in \mathcal R . 
\]
This implies that 
\[
w = \textbf{R}^T \ln ( \overline N). 
\]
  Hence $E=- \ln \overline N$ is a solution to \eqref{eq:energy}. 

    Let $\xi \in \mathcal A(0)$. Hence $\xi = \textbf{R} x $ for some $x \in \mathbb Z^{r/2}$. 
   By the definition of $\overline {\mathcal  E }$ we have that
   \[
   \tilde{\mathcal E} (\xi ) =\overline {\mathcal E}(x) = \sum_{i=1}^{r/2} x(i) \mathcal E(R_i) = w^T x.
   \]
  Now notice that if the vector $E\in \mathbb R^{N} $ satisfies \eqref{eq:energy}, then  
   \[
 \tilde{\mathcal E }(\xi)= w^T x = E^T \textbf{R} x =E^T \xi, \quad \forall x \in \mathbb R^{r/2}. 
   \]
\end{proof}
Notice that the solution to \eqref{eq:energy} is unique only if and only if $\mathcal M =\{ 0\}$
Otherwise we have that if $E_1 $ is a solution to \eqref{eq:energy}, then also the vector $E_2 \in \mathbb R^L $ defined as 
\[
E_2(i) = E_1(i) + \sum_{j=1}^L \mu_j m_j(i), \quad \forall i \in \Omega,
\]
where $ \{ m_j\}_{j=1}^L $ with $ m_j \in \mathbb R^N$ is a basis of $\mathcal M $ and where $\{ \mu_j \}_{j=1}^L   $ are the chemical potentials, is also a solution to equation \eqref{eq:energy}. 
In the following we will refer to each of the solutions of equation \eqref{eq:energy} as the \textit{energies induced by the kinetic system $(\Omega, \mathcal R, \mathcal K ) $}. 
\subsection{Global stability of the steady states for kinetic systems with the detailed balance property} \label{sec:global stability}

An important consequence of the detailed balance property is the existence of a natural free energy function that guarantees the existence of a globally stable positive steady state. 
In this section we firstly recall that, when the  detailed balance property holds, the steady states of \eqref{ODEs} can be written as a function of the energies $E$ introduced in Section \ref{sec:energy}. We then recall that a unique steady state exists in each stoichiometric compatibility class. Finally we recall the proof the stability of this steady state.
\begin{lemma} \label{lem:db special ss}
Assume that the kinetic system $(\Omega, \mathcal R, \mathcal K)$ satisfies the detailed balance property. 
Then $N^* =(N^*_i)_{i=1}^N$ is a steady state of the system of ODEs \eqref{ODEs} if and only if
      \begin{equation}      \label{stst when DB}
      N^*_i=  e^{- E(i)}, \quad i \in \{ 1, \dots, N\} 
      \end{equation}
      where $E \in \mathbb R^N $  is an energy induced by the kinetic system $(\Omega, \mathcal R, \mathcal K)$.
      \end{lemma}
\begin{proof}
Lemma \ref{lem:Db equivalence} implies that the detailed balance property holds if and only if every steady state $N^*$ of \eqref{ODEs} satisfies 
\[
0= J_R (N^*)= K_R \prod_{ j \in I(R)  } {N^*_j}^{-R(j)} - K_{- R }  \prod_{ j \in F(R)  } {N^*_j}^{R(j)}, \quad \forall R \in \mathcal R_s.
\]
This holds if and only if 
\[
\frac{K_R}{K_{-R}}= \frac{ \prod_{j \in F(R) }{N^*_j}^{R(j)} }{\prod_{j \in I(R) }{(N^*_j)}^{-R(j)} } = \prod_{j \in D(R) }{N^*_j}^{R(j)},  \quad \forall R \in \mathcal R  .
\]
Therefore $N^*$ is a steady state of \eqref{ODEs} if and only if it satisfies
\[
0 = \sum_{j \in \Omega }{ R(j) \ln(N^*_j) } + \ln (K_{-R}) - \ln (K_R), \quad \forall R \in \mathcal R. 
\]
As a consequence $N^*$ is a steady state of \eqref{ODEs} if and only if it satisfies
$ \textbf{R}^T \ln (N^*) + w =0$
 where the vector $w \in \mathbb R^{r/2} $ is defined as in \eqref{w}. 
 Therefore $N^*$ is a steady state if and only if
 \[
 \ln (N^*) =- E
 \] 
 where $E $ is a solution to \eqref{eq:energy}. 
\end{proof}

The following following lemma provides a useful way to compute the fluxes when the detailed balance property holds. 
\begin{lemma} \label{lem:flux db}
    Let $(\Omega, \mathcal K, \mathcal R)$ be a kinetic system that satisfies the detailed balance condition. Then for every $ R \in \mathcal R$ we have that the fluxes $J_R $ defined as \eqref{eq:fluxes} satisfy 
    \begin{equation} \label{flux DB}
    J_R(n) =K_R  \prod_{i \in I(R)} (n_i)^{- R(i) }  \left(1-  \prod_{i\in \Omega} n_i^{ R(i) } e^{ R(i) E(i) } \right), \forall n \in \mathbb R_*^N 
    \end{equation} 
    where $E \in \mathbb R^N $ is an energy of $(\Omega, \mathcal R, \mathcal K )$, i.e. is any solution to \eqref{eq:energy}.
    \end{lemma} 
    \begin{proof}
        Lemma \ref{lem:db special ss} guarantees that the vector $ \overline N =e^{-E}$ is a steady state of \eqref{ODEs} if and only if $E$ satisfies $\textbf{R}^T E= w $. 
    Hence for every $R \in \mathcal R_s $ we have that \begin{equation}
       \frac{K_{-R} }{K_{R} } = e^{ R^T E }
    \end{equation}
    As a consequence substituting this in \eqref{eq:fluxes} we obtain that
    \begin{align*}
        J_R(n) = K_R \left( \prod_{i \in I(R)} (n_i)^{- R(i) } - \prod_{i \in F(R)} n_i^{ R(i) } \prod_{i\in \Omega}  e^{ R(i)E(i) } \right) =  K_R  \prod_{i \in I(R)} (n_i)^{- R(i) }  \left(1-  \prod_{i\in \Omega} n_i^{ R(i) } e^{ R(i)E(i) } \right).
    \end{align*}
\end{proof}

\begin{proposition} \label{prop:global stability db big}
  Assume that the kinetic system $(\Omega, \mathcal R, \mathcal K )$ satisfies the detailed balance property.
  Then the system of ODEs \eqref{ODEs} has a unique positive steady state $N$ in each stoichiometric compatibility class and the steady state is globally stable. 
\end{proposition}
The proof of Proposition \ref{prop:global stability db big} can be found in \cite{feinberg2019foundations}. 
We review the main ideas behind that proof. 
Assume that the kinetic system $(\Omega , \mathcal R, \mathcal K )$ satisfies the detailed balance property. 
Then the function $F: \mathbb R_*^N \rightarrow \mathbb R$ defined as 
\begin{equation} \label{entropy}
F(n):= \sum_{j \in \Omega } n_j  \left(\log\left( n_j e^{E(j) }\right) - 1 \right)
\end{equation}
where $E$ is an energy of $(\Omega, \mathcal R, \mathcal K ) $, is the total \textit{free energy} associated with the kinetic system $(\Omega, \mathcal R, \mathcal K)$. 
The \textit{dissipation} $\mathcal D$ of free energy is defined as
\begin{equation} \label{dissipation}
\mathcal D(n):= -  \partial_t F (n).
\end{equation}
It is easy to check, using also Lemma \ref{lem:flux db}, that if a kinetic system satisfies the detailed balance property then 
\begin{align*}
\mathcal D(n) &=- \sum_{j \in \Omega } \partial_t n_j \log( n_j e^{ E(j) })= 
\sum_{j \in \Omega } \sum_{R \in \mathcal R_s } R(j) J_R(n) \log( n_j e^{ E(j) }) =  \sum_{R \in \mathcal R_s }  J_R(n) \log\left( \prod_{j \in \Omega}  n_j^{R(j)} e^{R(j) E(j) }\right) 
\\
&=  \sum_{R \in \mathcal R_s }  K_R  \prod_{i \in I(R)} (n_i)^{- R(i) }  \left(  \prod_{j\in \Omega} n_j^{ R(j) } e^{ R(j)E(j) } -1  \right) \log\left( \prod_{j \in \Omega}  n_j^{R(j)} e^{R(j) E(j) }\right).
\end{align*} 
Notice that $\mathcal D(n) \geq 0 $ for every $n \in \mathbb R_*^N$ because the function $(x-1) \log(x) \geq 0$ for every $x \geq 0$. Hence kinetic systems that satisfy the detailed balance property have a non-increasing free energy. 
  Moreover, we have that $\mathcal D(n)=0$ if and only if $ n = e^{- E} $. As a consequence of Lemma \ref{lem:db special ss}, we have that the steady states are the only minimizers of the free energy.
  Since for each stoichiometric compatibility class it is possible to prove that there exists a unique steady state (see \cite[Proposition 13.A.1, Corollary 13.A.3] {feinberg2019foundations}), the statement of Proposition \ref{prop:global stability db big} follows by noting that $F$ is a radially unbounded Lyapunov functional. 

Notice that the statement of Proposition \ref{prop:global stability db big} can be reformulated as follows. If the kinetic system $(\Omega, \mathcal R, \mathcal K ) $ satisfies the detailed balance property, then we have that for every initial concentration $n_0 $ there exists a unique positive steady state $\overline N$ such that $m^T \overline  N = m^T n_0 $ for every $m \in \mathcal M $ and such that the solution to the system of ODEs \eqref{ODEs} is such that $\lim_{t \to \infty} n(t)=\overline N$.

We conclude this section by introducing the definition of closed kinetic system. A closed kinetic system is a kinetic systems that does not exchange substances with the environment, hence satisfies the detailed balance condition and is conservative and does not contain sources and sinks.
\begin{definition}
    A kinetic system $(\Omega, \mathcal R, \mathcal K )$ is closed if it satisfies the detailed balance property, it is conservative and 
    \begin{equation} \label{no sources/sinks}
   I(R) \neq \emptyset \ \text{ and } \  F(R) \neq \emptyset \text{ for every } R \in \mathcal R. 
    \end{equation}
\end{definition}

\section{Reduction of a kinetic system} \label{sec:reduction}

 In this section we deal with kinetic systems that exchange substances with the environment. In particular, the assumption that we make in this section is that the time scale at which the exchange of chemicals between the kinetic systems and the environment takes place is much shorter than the time scale at which the reactions in the kinetic system take place. Hence we take the limit as $\alpha \rightarrow \infty $ in \eqref{eq:intro const fluxes}. 
 This justifies the assumption that certain substances in the kinetic system have constant concentration values. 

The plan for this section is the following. 
In Section \ref{sec:reduction of chem netowrks} we give the definition of reduced chemical network. 
In Section \ref{sec:reduced cycles} we study the relation between the cycles of the non-reduced kinetic system and the cycles of the reduced kinetic system.  
Finally in Section \ref{sec:reduction of kinetic systems} we define the rate function of the reduced kinetic system. This rate function will depend on the values of the frozen concentrations in the non-reduced kinetic system 

\subsection{Reduction of a chemical network}\label{sec:reduction of chem netowrks}
In this section we explain how we define the reduction of a chemical network. Before doing that we specify what we mean with projection of a vector. Assume that $v \in \mathbb Z^N$ is a vector and that $A \subset \Omega  $. Then we define the vector $ \pi_A v \in \mathbb R^{|A|} $ as 
\[
\pi_A v :=\left(  v(i)\right)_{i \in A}=(v( \min(A)), \dots,  v(\max A)).  
\]

\begin{definition} \label{def:reduced chemical netowork}
Let $(\Omega , \mathcal R)$ be a chemical network and let $U \subset \Omega$ be such that $U\neq \emptyset$ and let $V:= \Omega \setminus U$.
Let $v = |V | = |\Omega  \setminus U  |$. Without loss of generality we assume that $U=\{ d+1, \dots ,N \} $ and $V = \{1, \dots, d  \}$.
The set of the \textit{$U$-reduced reactions} ${ \mathcal R}_{V}$ is defined as as
\[
{\mathcal R }_{V} := \{ R \in \mathbb Z^{ |V|} \setminus \{ \textbf{0}_d\}  : R= \pi_V \overline{ R}, \text{ where } \overline{ R } \in \mathcal R \}.
\]
Then the chemical network $(V, {\mathcal R}_{V })$ is the \textit{$U$-reduced chemical network}  associated with the chemical network $(\Omega, \mathcal R)$.
\end{definition}
We denote with $\textbf{R}_{V}$ the matrix of the reactions associated with the set of reactions $\mathcal R_{V}$. 
Notice that the set of reactions $\mathcal R_{V}$ can contain reactions $R$ such that $I(R)=\emptyset$ or such that $F(R)=\emptyset$. 
 \begin{remark}
     If $(\Omega , \mathcal R) $ is bidirectional, then also $(V, {\mathcal R}_{V} ) $ is bidirectional. 
 \end{remark}
We stress that the number of the reactions in $\mathcal R$ can be strictly larger than the number of the reduced reactions $\mathcal R_V$. This can take place due two different reasons. 
Indeed, two reactions $R_1, R_2 \in \mathcal R $ with $R_1 \neq R_2$ and $ R_1 \neq - R_2  $ can be projected to the same reaction, i.e. it can happen that $\pi_V R_1 = \pi_V R_2$. Another possibility, is that a reaction $R \in \mathcal R$ is projected to zero as $\pi_V R =0. $
We illustrate these two possibilities in the following example. 
 \begin{example} \label{ex:1cycle reduced}
Consider the following set of reactions
\[
(1) + (5) \leftrightarrows (2), \quad  (1) + (6) \leftrightarrows (2), \quad (2) \leftrightarrows (3),\quad  (1) \leftrightarrows  (4), \quad (4)+ (5) \leftrightarrows (3), \quad (4) + (6) \leftrightarrows  (3). 
\]
In this case $\Omega=\{1,2,3,4,5,6 \}$ and the matrix of the reactions is given by
\[
\textbf{R}= 
\left( 
\begin{matrix}
&-1 & -1 & 0 & -1 & 0 & 0 \\
&1 & 1 & -1 & 0 & 0 & 0 \\
&0 & 0 & 1 & 0 & -1 & -1 \\
&0 &0 &0 &1 &1 &1 \\ 
&-1 &0 &0 &0 &0 &1 \\
&0 &-1 & 0& 0 &1 &0 
\end{matrix} \right). 
\]
Assume that $U:=\{ 5,6 \}$. Notice that $\pi_V R_1 =\pi_V \textbf{R} e_1  = \pi_V R_2 =\pi_V \textbf{R} e_2  $ as well as  $\pi_V R_5 =\pi_V \textbf{R} e_5 = \pi_V R_6 =\pi_V \textbf{R} e_6  $. 
The reduced chemical network is $(V, \mathcal R_V )$ where $V:=\{ 1,2,3,4 \}$ and the matrix associated to this set of reactions as in \eqref{matrix reactions} is given by 
\[ 
\textbf{R}_V:=\left( \begin{matrix}
 & -1 & 0 & -1 & 0  \\
 & 1 & -1 & 0 & 0 \\
 & 0 & 1 & 0 & -1  \\
 &0 &0 &1 &1 
\end{matrix} \right). 
\]
Consider instead $\overline U=\{ 1, 4\} $ in this case we have that $ \pi_{\overline V} R_4 = \pi_{\overline V }\textbf{R} e_4 =0$, hence the matrix of the reduced reactions is given by 
reactions is given by
\[
\textbf{R}_{\overline V}= 
\left( 
\begin{matrix}
&1 & 1 & -1  & 0 & 0 \\
&0 & 0 & 1 & -1 & -1 \\
&-1 &0 &0 &0&1 \\
&0 &-1 & 0&  1 &0 
\end{matrix} \right). 
\]
\end{example}
This example motivates the following definitions. 
 \begin{definition}[$1$-$1$-reduced reaction and reaction with zero-$V$-reduction] \label{def:11 and zero}
 Let $(\Omega, \mathcal R)$ be a chemical network and assume that $U$ and $V$ are as in Definition \ref{def:reduced chemical netowork}.  
 Let us consider the $U$-reduced chemical network $(V , \mathcal R_V )$ induced by the chemical network $(\Omega, \mathcal R)$. 
  We say that a reaction $R \in \mathcal R_V $ is a one-to-one reduced reaction if 
  \[
|\pi_V^{-1} (R) |=1 .
  \]
We say that a reaction $R \in \mathcal R$ has zero-$V$-reduction if $\pi_V R =0$. 
 \end{definition}
Notice that $ | \mathcal R_V |= | \mathcal R |$ if and only if every $ R \in \mathcal R $ is not a reaction that has a zero-$V$-reduction and every reaction in $\mathcal R_V $ is a $1$-$1$ reaction. Otherwise we have $|\mathcal R_V | < |\mathcal R|$. 

 Finally notice that the reduction described above decomposes the matrix of the reactions $\textbf{R} $ associated with $\mathcal R$ in two matrices $ \pi_U \textbf{R}\in \mathbb R^{(N-d) \times s}$ and $ \pi_V \textbf{R} \in \mathbb R^{d \times s }$ such that 
 \begin{equation} \label{two submatrices of reactions}
 \textbf{R}= \left(  \begin{matrix}
     \pi_V \textbf{R} \\
        \pi_U  \textbf{R}\\
 \end{matrix} \right). 
 \end{equation}
 Here the matrix $\pi_V \textbf{R}$ is the matrix wich has, as columns, the projections of the columns of $\textbf{R}$, i.e. the columns of $\pi_V \textbf{R}$ are of the form $\pi_V (\textbf{R} e_i)$.  
 We stress here that the matrix $\pi_V \textbf{R} $ is, in general, different from the matrix of the reduced reactions  $\textbf{R}_V$. 
 If every reaction in the reduced network is one-to-one and every reaction in the non reduced network does not have zero-$V$-reduction, then we have that the matrix $\pi_V \textbf{R} = \textbf{R}_V$.

\subsection{Reduction of the cycles} \label{sec:reduced cycles}
In the following sections we will provide precise conditions that guarantee that if the kinetic system $(\Omega, \mathcal R, \mathcal K ) $ satisfies the detailed balance property, then also the reduced kinetic system satisfies the detailed balance property. 
These conditions are topological conditions on the chemical networks $(\Omega, \mathcal R) $ and on the reduced chemical network $(V, \mathcal R_{V} )$. 
More precisely, these are conditions on the cycles $\mathcal C$ of the chemical system $(\Omega, \mathcal R) $ and on the the cycles $\mathcal C_{V}$ of the reduced chemical network. 
This is the reason why in this section we study in detail the relation between $\mathcal C$ and $ \mathcal C_V$.

\begin{figure}[H] 
\centering
\includegraphics[width=0.7\linewidth]{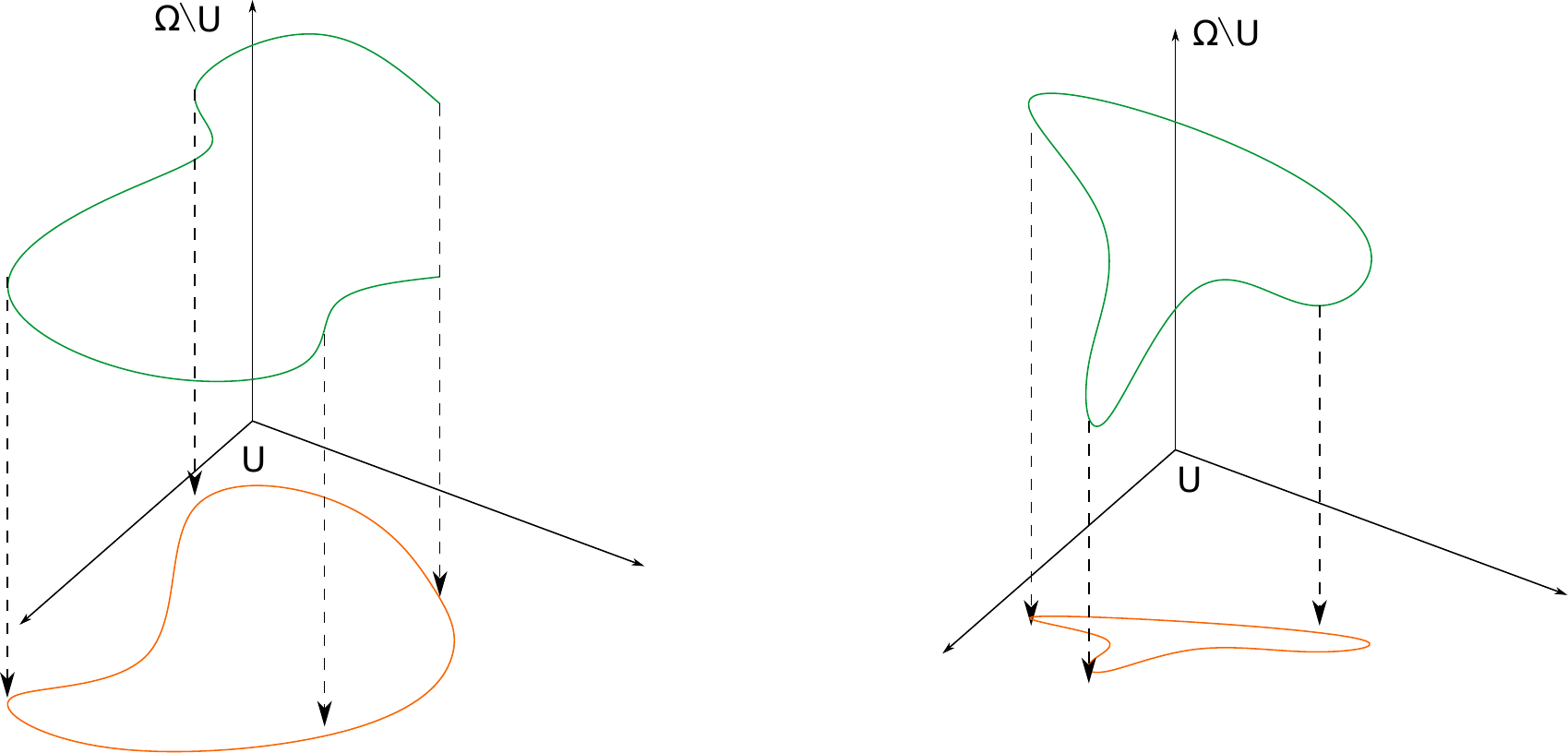}
\caption{
On the left we have a cycle in the reduced chemical network $ (U, \mathcal R_V)$ that is not a cycle in the non-reduced chemical network $ (\Omega , \mathcal R)$.  In the right we have a cycle in the reduced system that is also a cycle in the non reduced system. 
}\label{fig1}
\end{figure}

First of all we notice that if $|\mathcal R|= |\mathcal R_V |$ then we can compare directly the cycles and we have the following result. Notice that in this case we have that $\pi_V \textbf{R} =\textbf{R}_V$.
The following lemma states that every cycle of the reduced chemical network is a cycle also in the non reduced chemical network, see \eqref{ineq cylces}. 
Moreover, in the following lemma we state conditions on the matrix $\textbf{R}$ that are necessary and sufficient to have that the cycles of the reduced system and the cycles of the non reduced system are the same. 
\begin{lemma} \label{lem:cycles reduced 11}
Let $(\Omega, \mathcal R) $ be a chemical network with space of cycles $\mathcal C$. Let $U $ and $V$ be as in Definition \ref{def:reduced chemical netowork} and let $(V, {\mathcal R}_{V})$ be the $U$-reduced network.
Let $\mathcal C $ be the space of the cycles of the network $(\Omega, \mathcal R)$ and let $\mathcal C_{V}$ be the space of cycles of the $U$-reduced network $(V, {\mathcal R}_{V})$. 
Assume that $ |\mathcal R_V | = |\mathcal R|$.
Then
\begin{equation}\label{ineq cylces}
\mathcal C \subset {\mathcal C}_{V}. 
\end{equation}
Moreover, we have that $\mathcal C = {\mathcal C }_{V } $ if and only if $\ker(\textbf{R}_V)\subset \ker(\pi_U \textbf{R} ) $, where $ \pi_U \textbf{R} \in \mathbb R^{(L-d) \times s }$ is given by \eqref{two submatrices of reactions}.
\end{lemma}
\begin{proof}
First of all we prove that $ \mathcal C \subset {\mathcal C}_{V }$. 
    Let $c \in \mathcal C $. By definition we have that $\sum_{i=1}^{s} c(i) R_i =0$ where $\mathcal R_i =\textbf{R} e_i$ and where $s =|\mathcal R_s|$ where we recall that set $\mathcal R_s$ is defined as \eqref{non reversed set}.
By the definition of $\pi_V $ this implies that $ \sum_{i=1}^{s} c(i) \pi_V R_i=   0$. Hence \eqref{ineq cylces} follows. 

Assume now that ${\mathcal C }_{V }\subset \mathcal C $. Hence, by the definition of $\pi_U \textbf{R} $, it follows that $\pi_U \textbf{R}  c =0$ for every $c \in \mathcal C_{V } = \ker (\pi_V \textbf{R} )=  \ker( \textbf{R}_V )$. Hence $\ker(\textbf{R}_V  )\subset \ker(\pi_U \textbf{R} ) $. 
Assume instead that $ \mathcal C_{V } = \ker( \textbf{R}_V )\subset \ker(\pi_U \textbf{R} ) $. Hence for every $c \in \mathcal C_{V } $ we have that $\pi_U \textbf{R}  c =0$. As a consequence we have that $c \in \mathcal C$. 
\end{proof}

We have a similar statement when $|\mathcal R | > |\mathcal R_V|$. 
To formulate this statement it is convenient to introduce the definition of a suitable projection of cycles of the non reduced chemical network $(\Omega , \mathcal R) $ to the reduced chemical network $(V, \mathcal R_V) $. 
To this end notice that if $c \in \mathcal C $, then it holds that
\[ 
\pi_V\left( \sum_{i=1}^s R_i  c(i)\right) =\sum_{i=1}^s (\pi_V R_i) c(i) =0, 
\]
where $R_i = \textbf{R} e_i $ and $s=|\mathcal R_s|$.
Moreover,  
\[
\sum_{i=1}^s (\pi_V R_i ) c(i) = \sum_{j=1 }^v \overline R_j  \sum_{ \{ i: \overline R_j  =  \pi_V R_i \} } c(i)  
\]
where $\overline R_j=\textbf{R}_V e_i$ and $v$ is the number of reactions in $(\mathcal R_V)_s$ where the set $(\mathcal R_V)_s$ is defined as in \eqref{non reversed set} with respect to $\mathcal R_V$. 
This implies that the vector $p[c] \in  \mathbb R^{v}$ defined as
\begin{equation} \label{projection of cycles}
p[c](j):= \sum_{ \{ i: \overline R_j  =  \pi_V R_i \} }   c(i), \quad  j \in \{ 1, \dots, v\} 
\end{equation} 
is such that
$\textbf{R}_V p[c]=0$.
Therefore  $p[c] \in \mathcal C_V $ is the projection of the cycle $c \in \mathcal C$. 
To clarify the above definition we make the following example.
\begin{example} \label{ex:2cycle reduced}
Recall the chemical network $(\Omega , \mathcal R) $ of Example \ref{ex:1cycle reduced}. 
Notice that $\textbf{R} (1,1,2,-2,1,1) =0$.
Consider the reduction with respect to $U:=\{5,6\} $.
We now apply the projection map to the cycle $c=(1,1,2,-2,1,1)$ and obtain that $p[c]=(2,2,-2,2)$. 
Notice that this is a cycle of the reduced system. Indeed 
\[
\textbf{R}_V p[(1,1,2,-2,1,1)]= \left( \begin{matrix}
  -1 & 0 & -1 & 0  \\
  1 & -1 & 0 & 0 \\
  0 & 1 & 0 & -1  \\
 0 &0 &1 &1 
\end{matrix} \right)  \left( \begin{matrix} 2\\ 2 \\ -2 \\ 2  \end{matrix} \right) =0.
\]
\end{example}
We are now ready to state the  following Lemma \ref{lem:cycles reduced}, which is the analogous of Lemma \ref{lem:cycles reduced 11} when we have that $|\mathcal R |> |\mathcal R_V |$. 
\begin{lemma}\label{lem:cycles reduced}
Let $(\Omega, \mathcal R) $ be a chemical network with space of cycles $\mathcal C$. Let $U$ and $V$ be as in Definition \ref{def:reduced chemical netowork} and let $(V, {\mathcal R}_{V})$ be the corresponding $U$-reduced network. Let $\mathcal C_{V}$ be the space of cycles of $(V, {\mathcal R}_{V})$.
Then $p (\mathcal C):=\{ p[c] :c \in \mathcal C \}  = {\mathcal C }_{V } $ if $\ker( \pi_V\textbf{R} )\subset \ker(\pi_U \textbf{R}) $, where $ \pi_V\textbf{R} \in \mathbb R^{d \times s }$ and $ \pi_U \textbf{R} \in \mathbb R^{(N-d) \times s }$ are defined as in \eqref{two submatrices of reactions}.
\end{lemma}
\begin{proof}
    The fact that $p(\mathcal C) \subset \mathcal C_{V} $ follows by the definition of the map $p$. 
    We consider $c \in \mathcal C_{V } $ and assume that $ \ker( \pi_V \textbf{R}) \subset \ker( \pi_U \textbf{R}) $. Let $\tilde{c} \in \mathbb R^{|\mathcal R|}$ be such that $\tilde{c}(i) = c(i) $ for every $i \in \{ 1, \dots, |\mathcal R_V|\}$  and otherwise set $\tilde{c} (i)=0$. Notice that, upon reordering the reactions, by construction we have that $ \pi_V\textbf{R} \tilde{c} =0$. Hence $\textbf{R} \tilde{c} =0$ and $\tilde{c} \in \mathcal C $. Notice also that $c=p[\tilde{c}]$ hence the desired conclusion follows.
\end{proof}

We conclude this section by analysing the relation between the mass conservation laws of the chemical network $(\Omega, \mathcal R)$ and of the $U$-reduced chemical network $(V , \mathcal R_{V})$. 
\begin{lemma}
    Let $(\Omega, \mathcal R) $ be a chemical network with space of conservation laws $\mathcal M$. Let $U$ and $V $ be as in Definition \ref{def:reduced chemical netowork} and let $(V, {\mathcal R}_{V})$ be the corresponding $U$-reduced network. Let $\mathcal M_{V}$ be the space of conservation laws of the $U$-reduced network $(V, {\mathcal R}_{V})$. Then every $m \in \mathcal M_{V}$ is such that
\begin{equation}\label{ineq conserved}
\left(\begin{matrix}
    m \\
    \textbf{0}_{N-d} 
\end{matrix}\right)
\in \mathcal M.
\end{equation}
\end{lemma}

\subsection{Reduction of a kinetic system} \label{sec:reduction of kinetic systems}
In the previous section we explained how to reduce a chemical network $(\Omega , \mathcal R) $ to the network $(V , \mathcal R_V) $ where $V$ is as in Definition \ref{def:reduced chemical netowork}. 
We now explain how to construct the corresponding reduced kinetic system. 

Let $(\Omega, \mathcal R, \mathcal K ) $ be a kinetic system. Let $U, V$ be as in Definition \ref{def:reduced chemical netowork}. 
Let $n_U=(n_i)_{i \in U} $ be given. 
Consider the $U$-reduced kinetic system $(V, \mathcal R_V) $. We associate to this network the rate reaction function ${\mathcal K }[n_U]: {\mathcal R}_{V} \rightarrow \mathbb R $ where 
\begin{equation} \label{reduced rates}
{\mathcal K }[n_U](R) = K_R[n_U] =
 \sum_{ \{ \overline R = \pi_V^{-1} (R) \} } \hat{K}_{\overline R}
\end{equation}   
where 
\begin{equation} \label{Khat}
  \hat{K}_{\overline R } :=  K_{\overline R} \prod_{s \in U \cap I({\overline R}) } n_s^{-{\overline R}(s)}. 
\end{equation}
The kinetic system $(V, \mathcal R_{V}, {\mathcal K }[n_U])$ is the reduced kinetic system associated with the kinetic system $(\Omega , \mathcal R,  \mathcal K)$ where the concentrations of the substances in $U$ are constant and equal to $n_{  U }$. 

\begin{remark}
Assume that $(V, \mathcal R_{V}, {\mathcal K }[n_U])$ is the reduced kinetic system associated with the kinetic system $(\Omega , \mathcal R,  \mathcal K)$. 
If $\pi_V R\in\mathcal R_V $ is a $1$-$1$ reduction, then the identity \eqref{reduced rates} implies that 
\begin{equation} \label{reduced rates 11}
\frac{K_{- \pi_V R}[n_U]}{K_{\pi_V R}[n_U]}  = \frac{K_{- R}}{K_R } \prod_{s \in U} n_s^{{R}(s)} . 
\end{equation}
\end{remark}

\section{The detailed balance property of reduced kinetic systems} \label{sec:DB in the reduction}
The aim of this section is to study when a reduced kinetic system inherits the detailed balance property of the non reduced kinetic system $(\Omega , \mathcal R , \mathcal K )$.
In Section \ref{sec:detailed balance reduced equilibrium} we prove that if the frozen concentrations that appear in the cycles of the non reduced kinetic system are chosen at equilibrium values, then the reduced kinetic system satisfies the detailed balance property. 
In Section \ref{sec:detailed balance reduced} instead we prove that the reduction of a kinetic system that satisfies the detailed balance property "in general" does not satisfy the detailed balance property. Indeed we prove that it satisfies the detailed balance property in a robust manner if and only if either the frozen concentrations are chosen at equilibrium values, or the cycles in the non reduced and in the reduced system satisfy some specific properties that will be stated in detail later. 
In this paper, when we say that a property is robust we mean that the property remains true under small changes of the chemical rates or of the conserved quantities. 
The results obtained in Section \ref{sec:detailed balance reduced} are consistent with the fact that effective biochemical systems often do not satisfy the detailed balance property, as they are obtained by a reduction of a closed kinetic systems in which the concentration of certain substances is assumed to be constant in time.

\subsection{Reduced kinetic system with some concentrations fixed at equilibrium values} \label{sec:detailed balance reduced equilibrium}
In this section we provide sufficient conditions on $n_U $ that guarantee that if the kinetic system $(\Omega, \mathcal R, \mathcal K )$ satisfies the detailed balance property, then the reduced kinetic system $(V, \mathcal R_V, \mathcal K [n_U])$ satisfies the detailed balance condition. 
The main idea is that if the vector of the frozen concentrations $n_U $ is chosen at equilibrium concentrations, i.e. it is given by the projection on $U$ of a steady state $N $ of the system of ODEs \eqref{ODEs} induced by the kinetic system $(\Omega, \mathcal R, \mathcal K) $, then the reduced kinetic system $(V, \mathcal R_V, \mathcal K[n_U] ) $ satisfies the detailed balance property. 
However this statement can be made stronger, as it is sufficient to have that the concentration of the substances in $U$ are selected at equilibrium values if these substances appear in the cycles of the reduced system.

More precisely, we find that if every reaction that belongs to a cycle of the reduced kinetic system contains exactly the same substances as the corresponding non reduced reaction, i.e. if the substances in $U$ do not appear in the reactions  whose reduction belong to a cycle, then the detailed balance property of the reduced kinetic system follows directly by the detailed balance of the non reduced kinetic system for every value of frozen concentrations $n_U$.
 Instead, if we assume that some of the substances in the set $U$ appear in some of the reactions in $\mathcal R$ whose reduction belong to a cycle, then selecting the concentrations of these substances at equilibrium values guarantees that the reduced system satisfies the detailed balance property.  

In order to state these two results precisely it is useful to introduce the definition of the set $\mathcal D(\mathcal C_V)$, which is the set of the substances in $\Omega $ that belong to a reaction $R \in \mathcal R$ that, when projected to the reduced system, belongs to a cycle of the reduced system, i.e. belongs to $\mathcal C_V$. 
Hence, given a kinetic system $(\Omega, \mathcal R, \mathcal K)$ and its reduction  $(V, \mathcal R_V, \mathcal K[n_U])$ with space of cycles $\mathcal C_V $ we define the following set 
\begin{align}\label{eq:states in cycles}
    \mathcal D(\mathcal C_V ):=\{ i \in \Omega : \exists R \in \mathcal R \text{ s.t. } \pi_V R \in \mathcal C_V  \text{ and s.t. } R(i)\neq 0  \}.  
\end{align}
\begin{proposition} \label{prop: DB when equilibirum conc}
    Assume that the kinetic system $(\Omega, \mathcal R, \mathcal K ) $ satisfies the detailed balance property. Let $U \subset \Omega $ and $V:= \Omega \setminus U $. Consider the reduced kinetic system $(V , \mathcal R_{V} , \mathcal K [n_U] ) $ where $n_U:=\{ n_s  \}_{ s \in U } $ with $n_s>0$ for every $s \in U$. 
    \begin{enumerate}
        \item If $ U \cap \mathcal D(\mathcal C_V)= \emptyset $, then the kinetic system  $(V , \mathcal R_{V} , \mathcal K [n_U] ) $ satisfies the detailed balance property.
        \item If $ U \cap \mathcal D(\mathcal C_V) \neq \emptyset $ and
\begin{equation} \label{eq:conc at equilibirum}
n_s = e^{- E(s) } \quad  \forall s \in U \cap \mathcal D(\mathcal C_V)
\end{equation}
for an energy $E \in \mathbb R^N $ of the kinetic system $(\Omega, \mathcal R, \mathcal K )$. 
Then the kinetic system  $(V , \mathcal R_{V} , \mathcal K [n_U] ) $ satisfies the detailed balance property. 
    \end{enumerate}

\end{proposition}
\begin{proof}
We start by proving 1. 
The fact that $U \cap \mathcal D(\mathcal C_V) = \emptyset $ implies that for every $R \in \mathcal R$ such that $\pi_V R \in \mathcal C_V  $ it holds that $\pi_U R =0$. This implies that every reaction in $\mathcal C_V$ is a $1$-$1$ reaction. 
Moreover, the fact that for every $R \in \mathcal R$ such that $\pi_V R \in \mathcal C_V $ it holds that $\pi_U R =0$ implies that $\ker( \pi_V \textbf{R} ) \subset \ker( \pi_U \textbf{R} ) $. Indeed, let $x \in   \ker( \pi_V \textbf{R})$. Then by the definition of the projection $p $ we have that $p[x] \in \mathcal C_V$. Since the reactions in $\mathcal C_V $ are $1$-$1$ we have that, upon reordering the reactions, $x =(p[x], \textbf{0}_{r-v} )$ where we are using the notation $r= |\mathcal R|$ and $v= |\mathcal R_V |$.
Notice that since $\mathcal C_V \cap U = \emptyset$, then 
\[
\pi_U \textbf{R} x = \sum_{i=1 }^{r/2} x(i) \pi_U \textbf{R} e_i =  \sum_{i=1 }^{v/2} p[x](i) \pi_U R_i =0. 
\]
Hence Lemma \ref{lem:cycles reduced} implies that $p(\mathcal C)=\mathcal C_V$. 
To conclude the proof notice that the fact that $U \cap \mathcal D(\mathcal C_V) = \emptyset $ implies the rates of the reactions in the cycles are not modified. As a consequence for every $c \in p(\mathcal C )= \mathcal C_V $ it holds that 
\[
  \sum_{j=1}^{v/2 } c(j) \mathcal E \left( R_j  \right)= \sum_{j=1}^{v/2} c(j) \mathcal E \left( \pi_V  R_j  \right) =0 
\]
where $v=|\mathcal R_V|$ and where we are using the fact that since the reactions in $\mathcal C_V $ are $1$-$1$ then $p[c](i)=c(i)$ for every $i \in \{1, \dots, v\}$. 

We now prove 2. 
Consider $\overline R \in \mathcal R_V$ such that $\overline R \in \mathcal C_V$ and there exists a $R \in \pi_V^{-1} (\overline R) $ with $D(R) \cap U \neq \emptyset$. Let us compute $\frac{\hat{K}_{-R}}{\hat{K}_{R}}$ where $\hat{K} $ is defined as in \eqref{Khat}. 
Using the fact that $n_U$ satisfies \eqref{eq:conc at equilibirum} with respect to a solution $E$ of \eqref{eq:energy} we deduce that
\begin{align*}
    \frac{\hat{K}_{-R}}{\hat{K}_{R}} = \frac{ K_{-R} \prod_{i \in U \cap F(R) } n_i^{R(i) }} { K_{R} \prod_{i \in U \cap I(R) } n_i^{-R(i) }} = \frac{ K_{-R}}{K_R}\prod_{i \in U  } n_i^{R(i) } = \frac{ K_{-R}}{K_R} \exp \left( - \sum_{i \in U  } E(i)R(i) \right).
\end{align*}

Therefore we have that for $\overline R \in \mathcal R_V $ it holds that
\begin{align*}
\mathcal E( \overline R ) &= \log \left( \frac{ \sum_{R \in \pi^{-1}_V (\overline R) } \hat{K}_{-R}}{\sum_{R \in \pi^{-1}_V (\overline R) } \hat{K}_{R}}\right) = \log \left( \frac{ \sum_{R \in \pi^{-1}_V (\overline R) } \hat{K}_{R} \exp \left(\mathcal E (R) - \sum_{i \in U }    E(i)R(i)  \right) }{\sum_{R \in \pi^{-1}_V (\overline R) } \hat{K}_{R}}\right) \\
& =\log \left( \frac{ \sum_{R \in \pi^{-1}_V (\overline R) } \hat{K}_{R} \exp \left( \sum_{i \in \Omega }    E(i)  R(i) - \sum_{i \in U }    E(i) R(i)  \right) }{\sum_{R \in \pi^{-1}_V (\overline R) } \hat{K}_{R}}\right) \\
&=  \sum_{i \in V }    E(i) \overline R(i). 
\end{align*}
In the third equality we used the detailed balance property of $(\Omega , \mathcal R, \mathcal K)$ and in the last equality the fact that $R(i)=\overline R(i)$ for every $\overline R \in  \pi_V^{-1} (R)$ and for every $i \in V$. 
Consider now a $c \in \mathcal C_V $, then since $\textbf{R}_V c =0$ we have that 
\[
\sum_{j=1}^{v/2} c(j) \mathcal E( \overline R_j ) = c^T \textbf{R}_V^T \pi_V E =0. 
\]
As a consequence the kinetic system $(V , \mathcal R_{V} , \mathcal K [n_U] ) $ satisfies the detailed balance condition. 
\end{proof}
Proposition \ref{prop: DB when equilibirum conc} has a
consequence for systems with fluxes in a unique substance $s \in \Omega$. Clearly the statement of Proposition \ref{prop: DB when equilibirum conc} holds in the case in which $|U|=1$. However when $U \cap \mathcal D(\mathcal C_{\Omega \setminus \{ s\} } ) \neq \emptyset $ and $\mathcal M \neq \{ 0\} $ it is possible to find for every value of $n_s >0$ an energy $E$ that is such that $n_s = e^{-E(s)}$.  This is the content of the following corollary. 
\begin{corollary} \label{cor:one frozen conc db}
Let $(\Omega, \mathcal R, \mathcal K ) $ be a kinetic system that satisfies the detailed balance property.
If $\mathcal M \neq \{ 0\} $, then there exists a $s \in \Omega $ such that the reduced system $(V, \mathcal R_{V }  ,\mathcal{K}[n_s]) $, with $V:= \Omega \setminus \{ s \} $, satisfies the detailed balance condition for any value of $n_s>0$. 
\end{corollary}
\begin{proof}
 Since we assume that $\mathcal M \neq \{ 0\} $, then there exists a $s \in \Omega $ such that $m(s) \neq 0 $ for some $m \in \mathcal M $. 
Therefore, to conclude the proof we need to prove that, for any value of $n_s $, there exists a vector $E \in \mathbb R^{N}$ such that $\textbf{R}^T E = w $ and $ E(s)= -\log(n_s)  $. 
Assume by contradiction that $E(s) \neq  -\log(n_s)$ for every $E$ that satisfy $\textbf{R}^T E = w $.
Since there exists a $m \in \mathcal M $ such that $m(s)\neq 0 $ we can define the vector $v:= \frac{\log(n_s) + E(s)}{m(s)} m  $.  Define the vector $E^* := E- v $, then 
\[
R^T E^* = R^T E - R^T v =w \text{ and } E^*(s)= - \log(n_s). 
\]
This is a contradiction, therefore there exists a vector $E \in \mathbb R^{N}$ such that $\textbf{R}^T E = w $ and $ E(s)= -\log(n_s)  $. As a consequence the assumptions of Proposition \ref{prop: DB when equilibirum conc} hold, hence the system with fluxes $(\Omega , \mathcal R, \mathcal K , n_s) $ satisfies the detailed balance property. 
\end{proof}

\begin{remark}
    Notice that Proposition \ref{prop: DB when equilibirum conc}, combined with \ref{prop:global stability db big}, has in important consequence. Indeed if we assume that the set of the constant concentrations $n_U $ satisfy \eqref{eq:conc at equilibirum}, then the system of ODEs \eqref{ODEs} corresponding to the reduced kinetic system $(V, \mathcal R_V, \mathcal K [n_U]) $ has a unique steady state for every stoichiometric compatibility class. This steady state is globally stable. Notice that the compatibility class here is with respect to the stoichiometry $\mathcal S_V$. Similarly, Corollary \ref{cor:one frozen conc db} implies that under the assumption of the corollary we have that the reduced kinetic system $(V, \mathcal R_V, \mathcal K [n_s]) $ has a unique steady state for every compatibility class. This steady state is globally stable. 
\end{remark}

\subsection{Stability of the detailed balance property of the reduced system} \label{sec:detailed balance reduced}
In this section we study under which topological condition the reduced kinetic system $(V, \mathcal R_V, \mathcal K [n_U]) $ inherits the detailed balance property of the original kinetic system $(\Omega, \mathcal  R, \mathcal K) $ independently on the choice of the values of the frozen concentrations $n_U$. 

    \begin{theorem} \label{thm:same cycles implies db}
        Assume that the kinetic system $(\Omega, \mathcal R, \mathcal K ) $ satisfies the detailed balance property. Let $\mathcal C $ be the associated space of cycles.  
        Let $U $ and $V$ be as in Definition \ref{def:reduced chemical netowork} and let $(V, \mathcal R_{V}, {\mathcal K }[n_U]) $ be the $U$-reduced kinetic system $(V, \mathcal R_{V}, {\mathcal K }[n_U]) $. Let $\mathcal C_V $ be the space of the cycles associated with $(V, \mathcal R_{V}, {\mathcal K }[n_U]) $.
        Assume that 
        \begin{enumerate}
         
            \item 
        each reaction $R \in \mathcal C_V $ is $1$-$1$ (see Definition \ref{def:11 and zero}); 
        \item $ \mathcal C $ does not contain reactions with zero-$V$-reductions (see Definition \ref{def:11 and zero}); 
        \item 
$ p(   \mathcal C) = \mathcal C_{V}$, where we recall that the projection $p$ is defined as in \eqref{projection of cycles}. 
    \end{enumerate}
 Then $(V, \mathcal R_{V}, {\mathcal K } [n_U]  ) $ satisfies the detailed balance property for every $n_U = {( n_i )}_{ i \in U } $. 
    \end{theorem}
    Notice that the three conditions in the above theorem imply that all the cycles of the reduced and of the non reduced chemical systems are as in the right part of Figure \ref{fig1}.
\begin{proof} 
Since we are assuming that every $\overline R \in \mathcal C_V $ is $1$-$1$ we have that for every $\overline R_j \in \mathcal C_V $ there exists a unique $R_j $ such that $\pi_V R_j=\overline R_j$. Moreover since equality \eqref{reduced rates 11} holds, we deduce that 
    \[
     \mathcal E( \pi_V R_j)=   \log \left( \frac{ K_{- \pi_V R_j} [n_U] }{ K_{\pi_V R_j} [n_U]} \right)  =         \log \left( \frac{  K_{ - R_j} }{ K_{R_j}} \right) - \sum_{s \in U } R_j(s) \log( n_s) = \mathcal E(R_j) - \sum_{s \in U } R_j(s) \log( n_s). 
    \]
    Therefore for every $c \in \mathcal C_V$ we have that
    \[
      \sum_{j =1}^{v/2} c(j) \mathcal E(\pi_V R_j)=  \sum_{j =1}^{v/2} c(j) \mathcal E(R_j) - \sum_{j =1}^{v/2} c(j) \sum_{s \in U } R_j(s) \log( n_s)
    \]
    where $v:=|\mathcal R_V|$. 
    We now use the fact that $p(\mathcal C)=\mathcal C_V$ to deduce that there exists a $\overline c \in \mathcal C$ such that $p(\overline c)=c $. 
   Since $\mathcal C $  does not contain reactions with zero-$V$-reduction and $\mathcal C_V$ does not contain reactions that are not $1$-$1$ reactions, we deduce that (upon reordering of the reactions) the cycle $\overline c $ has the following form $ \overline c=(\textbf{0}, c)$. 
    Since $(\Omega, \mathcal R, \mathcal K )$ satisfies the detailed balance property it follows that 
\[ 
 0= \sum_{j =1}^{r/2} \overline{c} (j)  \mathcal E(R_j) =   \sum_{j =1}^{v /2} c(j)  \mathcal E(R_j) . 
  \] 
  Therefore we deduce that 
  \[
     \sum_{j =1}^{r/2} c(j) \mathcal E(\pi_V R_j)=   - \sum_{s \in U } \sum_{j =1}^{r/2} c(j) R_j(s) \log( n_s) =0
\] 
Where the last equality comes from the fact that $c\in \mathcal C_V $, hence $\sum_{j =1}^{r/2} c(j) R_j(s) =0$ for every $s \in U $. 
\end{proof}
\begin{remark}
Notice that the requirements 1., 2. and 3. in Theorem \ref{thm:same cycles implies db} are all topological requirements that depend only on the chemical network $(\Omega, \mathcal R)$ and do not depend on the choice of the rate function. As a consequence, when the assumptions of Theorem \ref{thm:same cycles implies db} are satisfied, then the detailed balance property of the reduced kinetic system is stable under perturbations of the reaction rates of the non reduced kinetic system $(\Omega, \mathcal R, \mathcal K )$. 
\end{remark}
Notice that in Theorem \ref{thm:same cycles implies db} we have three assumptions. These three assumptions are all needed in order to have a reduced kinetic system that satisfies the detailed balance property in a stable manner. We start by analysing what happens when assumptions 1. and 2. of Theorem \ref{thm:same cycles implies db}  holds, but $p (\mathcal C) \neq \mathcal C_V $. 

Since in the following proposition we deal with a perturbation of the rates it is convenient to introduce the definition of norm of a rate function. 
Let $\mathcal K $ be a rate function, then 
\begin{equation} \label{eq norm rate}
    \| \mathcal K  \|:= \max_{R \in \mathcal R} |\mathcal K (R) | . 
\end{equation}

\begin{proposition} \label{prop: instability of (DB) different cycles}
    Assume that the kinetic system $(\Omega, \mathcal R, \mathcal K ) $ satisfies the detailed balance property. Let $\mathcal C $ be the associated space of cycles.  
       Let $U $ and $V$ be as in Definition \ref{def:reduced chemical netowork} and let $(V, \mathcal R_{V}, {\mathcal K }[n_U]) $ be the $U$-reduced kinetic system $(V, \mathcal R_{V}, {\mathcal K }[n_U]) $. 
        Assume that $p(\mathcal C) \neq  \mathcal C_{V} $ and each reaction in $\mathcal C_V$ is $1$-$1$ and $ \mathcal C $ does not contain reactions with zero-$V$-reductions. 
        Assume that the reduced kinetic system satisfies the detailed balance property. Then we have two possibilities: 
        \begin{enumerate}
          \item either the concentrations $n_U:=( n_i)_{i \in U } $ satisfy \eqref{eq:conc at equilibirum} for a vector of energies $E \in \mathbb R^N $ satisfying \eqref{eq:energy}; 
            \item or, alternatively, we have that for every $\delta >0 $ there exists a map $ \mathcal K_\delta : \mathcal R \rightarrow \mathbb R_+$ such that
            \[
            \|\mathcal K- \mathcal K_\delta \| < \delta 
            \]
            and such that $(\Omega, \mathcal R, \overline {\mathcal K}_\delta [n_U])$ does not satisfy the detailed balance condition. 
            \end{enumerate}
\end{proposition}
\begin{proof}
Assume that the detailed balance property of the reduced system holds. Let $c \in \mathcal C_{V } \setminus p(\mathcal C)  $.  Then we have that
\begin{equation} \label{eq:reduced db}
\prod_{j=1}^{v/2} \left(\frac{K_{\overline R_j} [n_U]}{K_{-\overline R_j} [n_U]} \right)^{c(j)} =1 ,
\end{equation}
where $ \{ \overline R_j\}_{j=1}^{v/2} = \mathcal R_V$.
Since we are assuming that all the reactions in the cycles in $\mathcal C_V $ are one-to-one, the definition of $\mathcal K[n_U]$ together with \eqref{eq:reduced db} implies that
\begin{align*}
\sum_{j=1}^{v/2}  c(j) \log\left( \frac{K_{R_j}}{K_{-R_j}} \right) = \sum_{j=1}^{v/2} \sum_{s \in U }  c(j) R_j(s) \log(n_s),  
\end{align*}
where $R_j =\pi_V^{-1} (\overline R_j)$.
Notice that since the detailed balance property holds we have that $\log\left( \frac{K_R}{K_{-R} }\right) = - \sum_{i \in \Omega} R(i) E(i)  $ where $E$ is any solution to \eqref{eq:energy}. Hence using the fact that $c \in \mathcal C_V $ we deduce that 
\begin{align*}
  \sum_{j=1}^{v/2}  c(j) \log\left( \frac{K_{R_j}}{K_{-R_j}} \right) = - \sum_{j=1}^{v/2}  c(j)  \sum_{i \in \Omega} R_j (i) E(i)  =  - \sum_{j=1}^{v/2}  c(j)  \sum_{i \in U } R_j (i) E(i), 
\end{align*}
with $v= |\mathcal R_V|$. 
Hence we obtain that the detailed balance property of the reduced system holds if and only if for every $c \in \mathcal C_V$ we have that
\begin{align} \label{eq:db perturbative different cycles}
- \sum_{j=1}^{v/2}  c(j)  \sum_{i \in U \cap \mathcal D(\mathcal C_V) } R_j (i) E(i) &= - \sum_{j=1}^{v/2}  c(j)  \sum_{i \in U } R_j (i) E(i) = \sum_{j=1}^{v/2} \sum_{i \in U  }  c(j) R_j(i) \log(n_i) \\
&= \sum_{j=1}^{v/2} \sum_{i \in U  \cap \mathcal D(\mathcal C_V)  }  c(j) R_j(i) \log(n_i). \nonumber    
\end{align}
Notice that since $c \in \mathcal C_V \setminus p(\mathcal C) $ we have that $ \sum_{j=1}^{v/2}   c(j) R_j(\ell) \neq 0 $ for at least an $\ell \in U  \cap \mathcal D(\mathcal C_V)$. 
Assume that $n_U $ has not the form \eqref{eq:conc at equilibirum}, we want to prove that it is possible to construct a perturbation of the rate function $\mathcal K_\delta $ that is such that \eqref{eq:db perturbative different cycles} fails for these rates, hence the kinetic system $(V, \mathcal R_V, \mathcal K_\delta[n_U]) $ does not satisfy the detailed balance property.
Then we can define the perturbed rate function $\mathcal K_\delta : \mathcal R \rightarrow \mathbb R_+ $ as $\mathcal K_\delta (R)=\mathcal K (R)$ for every $R \in \mathcal R_s $ and  \[
\mathcal K_\delta (-  R)  =\mathcal K_\delta ( R) e^{\sum_{i \in \Omega }  R(i) (E(i)+ E_\delta (i))} =\mathcal K(-R)  e^{\sum_{i \in \Omega }  R(i)  E_\delta (i)} 
\] 
for every $R \in \mathcal R_s$. Here we have that $E_\delta (i)=0$ for every $i \neq \ell $ and $E_\delta (\ell) <  \frac{\delta}{\max_{R \in \mathcal R_s} |R(\ell)|\max_{ R \in \mathcal R } \mathcal K(R)}$. 
Notice that by construction $E+E_\delta$ is an energy of the perturbed kinetic system $(\Omega , \mathcal R, \mathcal K_\delta )$ and the perturbed rates $\mathcal K_\delta$ satisfy the assumptions of the theorem indeed we have that by construction $(\Omega, \mathcal R, \mathcal K_\delta) $ satisfies the detailed balance property and moreover we have that 
\[ 
\| \mathcal K - \mathcal K_\delta  \| \leq \max_{ R \in \mathcal R } |\mathcal K(R)- \mathcal K_\delta (R)| \leq \max_{R\in \mathcal R} \mathcal K (R) \max_{ R \in \mathcal R } |1 - e^{E_\delta (\ell )  R(\ell) } | \leq\max_{ R \in \mathcal R } \mathcal K(R) \max_{ R \in \mathcal R } |E_\delta (\ell )  R(\ell)  | <  \delta.
\]
Notice that to obtain the above inequalities we are using the fact that for sufficiently small $x >0$ we have that $1- e^{x } < x $.  
Moreover, equality \eqref{eq:db perturbative different cycles} fails for the kinetic system $(\Omega , \mathcal R, \mathcal K_\delta )$. 
Indeed we have that 
\begin{align*}
 \sum_{j=1}^{v/2}  c(j)  \sum_{i \in U \cap \mathcal D(\mathcal C_V) } R_j (i) [E(i)+ E_\delta(i)  ] + \sum_{j=1}^{v/2} \sum_{i \in U  \cap \mathcal D(\mathcal C_V)  }  c(j) R_j(i) \log(n_i)=    E_\delta(\ell ) \sum_{j=1}^{v/2}  c(j) R_j (\ell)  \neq 0   
\end{align*}
that, therefore, the kinetic system $(\Omega , \mathcal R, \mathcal K_\delta[n_U] )$ does not satisfy the detailed balance property. 
\end{proof}
In the following example, we show that the assumption that the reactions in $\mathcal C_V$ are $1$-$1$ reactions is needed for Theorem \ref{thm:same cycles implies db} to hold.  
\begin{example} \label{ex:not 11 no Db}
Let $(\Omega, \mathcal R)$ be the network in Example \ref{ex:1cycle reduced}. 
As already anticipated in Example \ref{ex:2cycle reduced}, the space of the cycles of this chemical network is 
\[ 
\mathcal C = \operatorname{span} \{ c_1, c_2 \}  
\]
with $c_1=(1,0,1,-1,0,1)$ and $c_2 =(0,1,1,-1,1,0)$.
Consider the $U$ reduced kinetic system with $U=\{5,6\}$. Then we obtain that
\[
\mathcal C_{V} =  \operatorname{span}(c)
\]
where $c=(1,1,-1,1)$.
Notice that $p[c_1] = p[c_2] = c $.
As a consequence in this example it holds that $p (\mathcal C) = \mathcal C_{V } $. 
Assume that 
\[  K_{R_1}  = K_{-R_1} =  K_{R_2}  =K_{-R_2} =  K_{R_3}= K_{-R_3}  = K_{R_4}  = K_{-R_4} = K_{-R_5}  = K_{R_5} =K_{R_6}  =K_{-R_6} =1. 
\]
Clearly for every cycle $c \in \mathcal C$ we have that 
\[
\prod_{j=1}^{6} \left( \frac{  K_{R_j} }{K_{-R_j}} \right)^{c(j) } =1.
\]
Hence the non reduced kinetic system $(\Omega, \mathcal R, \mathcal K )$ satisfies the detailed balance property. 
Let $\tilde{R}_i$, for $i=1,2,3,4$ be the reactions of the reduced system computed in Example \ref{ex:1cycle reduced}.
Equality \eqref{reduced rates} implies that the rates of the reduced kinetic system are given by  
\[
K_{\tilde{R}_4}[n_U]= K_{R_5} + K_{R_6} = 2
\ \text{ and } \ K_{-\tilde{R}_4}[n_U] =  K_{-R_{6}} n_5 +K_{-R_{5}} n_6  =  n_5 + n_6,
\] 
similarly we have that
\[
K_{\tilde{R}_1}[n_U]= n_5+n_6 \  \text{ and } \  K_{-\tilde{R}_1}[n_U] =2. 
\] 
In this case it is easy to see that
 for every cycle $c \in \mathcal C_V$ we have
\[
\prod_{j=1}^{4} \left( \frac{  K_{\tilde R_j} [n_U]}{K_{-\tilde R_j} [n_U]} \right)^{c(j) } =1. 
\]
Hence the reduced kinetic system satisfies the detailed balance property. 
Assume instead that 
\[  K_{R_1}  = K_{-R_1} =  K_{R_2}  =K_{-R_2} =  K_{R_3}= K_{-R_3}  = K_{R_4}  = K_{-R_4} = K_{-R_5}  = K_{R_5}  =1. 
\]
while 
\[ 
 K_{R_6} = 2, \   K_{-R_6} = 2. 
\] 
Notice that by construction for every cycle $c \in \mathcal C$ we have that 
\[
\prod_{j=1}^{6} \left( \frac{  K_{R_j} }{K_{-R_j}} \right)^{c(j) } =1.
\]
Hence $(\Omega, \mathcal R, \mathcal K)$ satisfies the detailed balance property.
However, in this case we have that 
\[
K_{\tilde{R}_4}[n_U]= K_{R_5} + K_{R_6} = 3
\ \text{ and } \ K_{-\tilde{R}_4}[n_U] = ( K_{-R_{6}} n_5 +K_{-R_{5}} n_6)  =  n_6 + 2 n_5  
\] 
while $K_{\tilde{R}_1}[n_U]= n_5+n_6 $ while $K_{-\tilde{R}_1}[n_U] =2 $.
Therefore for $c=(1,1,-1,1)$ we have that 
\[
\prod_{j=1}^{4} \left( \frac{  K_{R_j}[n_U] }{K_{-R_j} [n_U]} \right)^{c(j) } = \frac{3 (n_5+n_6) }{2(n_6+ 2n_5)  } \neq 1
\]
for $n_5 \neq n_6 $.  Hence, the detailed balance property in this case holds only if the concentrations of $n_5, n_6 $ are chosen at equilibrium values, i.e. if they are such that $n_5=N_5$ and $n_6=N_6$ where $N$ is the steady state of the kinetic system $(\Omega, \mathcal R, \mathcal K ) $. 
\end{example}
The example above motivates the following proposition. 
\begin{proposition} \label{prop: instability of (DB) no one to one}
    Assume that the kinetic system $(\Omega, \mathcal R, \mathcal K ) $ satisfies the detailed balance property.
    Let $U $ and $V$ be as in Definition \ref{def:reduced chemical netowork}.
    Assume moreover that there exists at least one reaction $R \in \mathcal C_V $ that is not a $1$-$1$ reaction.
        Assume that the reduced system $(V, \mathcal R_V, \mathcal K [n_U]) $ satisfies the detailed balance.  
        Then we have two possibilities: 
        \begin{enumerate}
            \item either the concentrations $n_U:=\{ n_i\}_{i \in U } $ satisfy \eqref{eq:conc at equilibirum}; 
            \item or, alternatively,  we have that for every $\delta >0 $ there exists a map $ \mathcal K_\delta : \mathcal R \rightarrow \mathbb R_+$ such that
            \[
            \|\mathcal K- \mathcal K_\delta \| < \delta 
            \]
            and such that $(\Omega, \mathcal R, {\mathcal K}_\delta [n_U])$ does not satisfy the detailed balance condition. 
        \end{enumerate}
\end{proposition}
\begin{proof} 
Assume that the kinetic system  $(V, \mathcal R_V, \mathcal K[n_U] ) $ satisfies the detailed balance property. Then we must have that if $c \in \mathcal C_V$, then
\[
\sum_{j=1}^{v/2} c(j) \mathcal E(\overline R_j)=0
\]
where $\{ \overline  R_j\}_{j=1}^v =\mathcal R_V $. Assume that $c \in \mathcal C_V$ contains at least a reaction that is not one to one. 
Notice that for every $j \in \{1, \dots, v/2\} $ we have that 
\begin{align*}
    \mathcal E(\overline R_j)=\log \left( \frac{\sum_{R_{k_j} \in \pi_V^{-1} (\overline R_j)} \hat{K}_{- R_{k_j}}}{\sum_{R_{k_j} \in \pi_V^{-1} (\overline R_j) } \hat{K}_{ R_{k_j}}}\right) 
\end{align*}
where $\hat{K}_{R_{k_j}} = K_{R_{k_j}} \prod_{i \in U \cap I(R_j)} n_i^{-R_{k_j}(i)}$ and $\hat{K}_{-R_{k_j}} = K_{-R_{k_j}} \prod_{i \in U \cap F(R_{k_j})} n_i^{R_{k_j}(i)}$. 
The detailed  balance property of $(\Omega, \mathcal R, \mathcal K ) $ implies that $K_{-R_{k_j}} = K_{R_{k_j}}  e^{\sum_{i \in \Omega }R_{k_j}(i) E(i)}$ for every $E$ solution to \eqref{eq:energy}. 
Hence 
\[
\frac{\hat{K}_{-R_{k_j}}}{\hat{K}_{R_{k_j}}} = e^{\sum_{i \in \Omega }R_{k_j}(i) E(i)} \prod_{i \in U } n_i^{R_{k_j}(i)}.
\]
We deduce that 
\begin{align*}
    \mathcal E(\overline R_j)&=\log \left( \frac{\sum_{R_{k_j} \in \pi_V^{-1} (\overline R_j)} \hat{K}_{ R_{k_j}}e^{\sum_{i \in \Omega }R_{k_j}(i) E(i)} \prod_{i \in U } n_i^{R_{k_j}(i)}}{\sum_{R_{k_j}\in \pi_V^{-1} (\overline R_j) } \hat{K}_{R_{k_j}}}\right) \\
    &= \sum_{k \in V} \overline R_j(k) E(k) +\log \left( \frac{\sum_{R_{k_j} \in \pi_V^{-1} (\overline R_j)} \hat{K}_{ R_{k_j}} \prod_{i \in U }  (n_ie^{E(i)})^{R_{k_j}(i)}}{\sum_{R_{k_j}\in \pi_V^{-1} (\overline R_j) } \hat{K}_{R_{k_j}}}\right).
\end{align*}
The detailed balance property of the reduced system implies that
\begin{align*}
  0=  \sum_{j=1}^{v/2} c(j) \mathcal E(\overline R_j) &=   \sum_{k \in V} E(k) \sum_{j=1}^{v/2} c(j)\overline R_j(k)  +     \sum_{j=1}^{v/2} c(j) \lambda_j= \sum_{j=1}^{v/2} c(j) \lambda_j,
\end{align*}
where 
\[ 
\lambda_j:= \log \left( \frac{\sum_{R_{k_j} \in \pi_V^{-1} (\overline R_j)} \hat{K}_{ R_{k_j}} \prod_{i \in U }  (n_ie^{E(i)})^{R_{k_j}(i)}}{\sum_{R_{k_j} \in \pi_V^{-1} (\overline R_j) } \hat{K}_{ R_{k_j}}}\right)
\]
We have two options, either $n_U $ is given by \eqref{eq:conc at equilibirum}, or alternatively, 
\[
\prod_{i \in U }  (n_ie^{E(i)})^{R_{k_j}(i)} = \alpha_{k_j} \neq 1 \text{ for some } j \in \{1, \dots, v/2\} 
\]
and 
\[ 
 \sum_{j=1}^{v/2} c(j) \log \left( \frac{\sum_{R_{k_j} \in \pi_V^{-1} (\overline R_j)} \hat{K}_{ R_{k_j}} \alpha_{k_j} }{\sum_{R_{k_j} \in \pi_V^{-1} (\overline R_j) } \hat{K}_{ R_{k_j}}}\right) =0.
\]
Without loss of generality we can assume that $\ell \in \{ 1, \dots, v/2\} $ is such that $c(\ell) \neq 0$ and $\alpha_{k_\ell} \neq 1 $. Consider a reaction $R_{\overline k_\ell } \in \pi^{-1}_V (R_\ell) $, and define the rate function 
${\mathcal K}_\delta  : \mathcal R \rightarrow \mathbb R_*$ such that 
\[
{\mathcal K}_\delta(R)=\mathcal K (R) \text{ for every } R \neq R_{\overline k_\ell }, - R_{\overline k_\ell } \text{ and }  {\mathcal K}_\delta ( R_{\overline k_\ell })=\mathcal K (R_{\overline k_\ell }) + \delta,  \quad  {\mathcal K}_\delta (- R_{\overline k_\ell })=\mathcal K_\delta (R_{\overline k_\ell }) e^{ \sum_{ i \in \Omega } R_{\overline k_\ell }(i) E(i) }
\]
for $\delta >0$.
Notice that by construction we have that the kinetic system $(\Omega , \mathcal R, \mathcal K_\delta  ) $ satisfies the detailed balance property and that $\| \mathcal K - \mathcal K_\delta \| < \delta $. Notice also that the energies of the unperturbed kinetic system $(\Omega , \mathcal R , \mathcal K) $ are also energies for  $(\Omega , \mathcal R , \mathcal K_\delta) $. 
We denote with $\hat{K}^{(\delta)}_{ R_j}$ the rates induced by $\mathcal K_\delta$ via \eqref{Khat}. 
Then 
\begin{align*}
 \sum_{j=1}^{v/2} c(j) \log \left( \frac{\sum_{R_{k_j }\in \pi_V^{-1} (\overline R_j)} \hat{K}^{(\delta)}_{ R_{k_j }} \alpha_{k_j } }{\sum_{R_{k_j } \in \pi_V^{-1} (\overline R_j) } \hat{K}^{(\delta)}_{ R_{k_j }}}\right) &= \sum_{j=1, j \neq \ell }^{v/2} c(j) \log \left( \frac{\sum_{R_{k_j } \in \pi_V^{-1} (\overline R_j)} \hat{K}_{ R_{k_j }} \alpha_{k_j } }{\sum_{R_{k_j } \in \pi_V^{-1} (\overline R_j) } \hat{K}_{ R_{k_j }}}\right) \\
 &+ c(\ell ) \log \left( \frac{\sum_{ R_{k_\ell}  \in \pi_V^{-1} (\overline R_\ell)} \hat{K}^{(\delta)}_{ R_{ k_\ell }} \alpha_{k_\ell } }{\sum_{R_{k_\ell} \in \pi_V^{-1} (\overline R_\ell) } \hat{K}^{(\delta)}_{ R_\ell}}\right). 
\end{align*}
We now take the derivative in $\delta$ and obtain that 
\begin{align*}
 \frac{d}{d\delta} \sum_{j=1}^{v/2} c(j) \log \left( \frac{\sum_{R_{k_j }\in \pi_V^{-1} (\overline R_j)} \hat{K}^{(\delta)}_{ R_{k_j }} \alpha_{k_j } }{\sum_{R_{k_j } \in \pi_V^{-1} (\overline R_j) } \hat{K}^{(\delta)}_{ R_{k_j }}}\right) = c(\ell ) \alpha_{\overline k_\ell } & \neq 0. 
\end{align*}
Since we have that $\sum_{j=1}^{v/2} c(j)  \left( \frac{\sum_{R_{k_j }\in \pi_V^{-1} (\overline R_j)} \hat{K}^{(0)}_{ R_{k_j }} \alpha_{k_j } }{\sum_{R_{k_j } \in \pi_V^{-1} (\overline R_j) } \hat{K}^{(0)}_{ R_{k_j }}}\right) =0$, this implies that 
\[
 \sum_{j=1}^{v/2} c(j)  \left( \frac{\sum_{R_{k_j }\in \pi_V^{-1} (\overline R_j)} \hat{K}^{(\delta)}_{ R_{k_j }} \alpha_{k_j } }{\sum_{R_{k_j } \in \pi_V^{-1} (\overline R_j) } \hat{K}^{(\delta)}_{ R_{k_j }}}\right)  \neq 0.
\]
Therefore
$(\Omega , \mathcal R_V ,{\mathcal K}_\delta [n_U]) $ does not satisfy the detailed balance property. 
\end{proof}
In the following example we show that also the assumption that the reactions in the cycles of the non-reduced kinetic system have non zero projection in the reduced system is needed in Theorem \ref{thm:same cycles implies db}. 
\begin{example}\label{ex:zero no Db}
    Let $(\Omega, \mathcal R)$ be the network in Example \ref{ex:1cycle reduced} and Example \ref{ex:not 11 no Db}. 
If
\[  K_{R_1}  = K_{-R_1} =  K_{R_2}  =K_{-R_2} =  K_{R_3}= K_{-R_3}  = K_{R_4}  = K_{-R_4} = K_{-R_5}  = K_{R_5} =K_{R_6}  =K_{-R_6} =1, 
\]
then $(\Omega , \mathcal R) $ satisfies the detailed balance property. Consider the reduced system induced by $\overline U=\{ 1, 4\} $.  Notice that, as shown in Example \ref{ex:1cycle reduced}, we have that $\pi_{\overline V} R_4=0$. As a consequence Assumption 2. in Theorem \ref{thm:same cycles implies db} fails. 
Let $\tilde{R}_i$, for $i=1,2,3,4,5$ be the reactions of the reduced system computed in Example \ref{ex:1cycle reduced}.
Equality \eqref{reduced rates} implies that 
\[
K_{\tilde{R}_1}[n_U]= n_1,\ K_{\tilde{R}_2}[n_U]= n_1, \  
 K_{-\tilde{R}_4}[n_U] =n_4, \ K_{-\tilde{R}_5}[n_U] =n_4
\] 
while the rest of the rates are equal to $1$. 
A cycle of the reduced system is $c=(1,1,2,1,1)$. 
Notice that 
\[
\prod_{j=1}^5 \left(\frac{K_{- \tilde R_j}}{K_{\tilde R_j}} \right)^{c(j)} =1 
\]
is true only when $n_1=n_4$. This means that the detailed balance property of the reduced system holds only if $n_1=n_4$. 
\end{example}
Motivated by the example above we prove the following statement. 
\begin{proposition} \label{prop: instability of (DB) one to one and zeros}
    Assume that the kinetic system $(\Omega, \mathcal R, \mathcal K ) $ satisfies the detailed balance property.
    Let $U $ and $V$ be as in Definition \ref{def:reduced chemical netowork}.
    Assume moreover that every reaction $R \in \mathcal C_V $ is a $1$-$1$ reaction. Moreover, we  assume that there exists at least a cycle $c_V \in \mathcal C_V$ that is such that every $c \in \mathcal C$ satisfying $p[c]= c_V$ contains at least one reaction with zero-$V$-reduction. 
        Assume that the reduced system $(V, \mathcal R_V, \mathcal K [n_U]) $ satisfies the detailed balance. 
        Then we have to possibilities: 
        \begin{enumerate}
            \item either the concentrations $n_U:=\{ n_i\}_{i \in U } $ satisfy \eqref{eq:conc at equilibirum}; 
            \item or, alternatively,  we have that for every $\delta >0 $ there exists a map $ \mathcal K_\delta : \mathcal R \rightarrow \mathbb R_+$ such that
            \[
            \|\mathcal K- \mathcal K_\delta \| < \delta 
            \]
            and such that $(\Omega, \mathcal R, {\mathcal K}_\delta [n_U])$ does not satisfy the detailed balance condition. 
        \end{enumerate}
\end{proposition}
\begin{proof}
    By assumption we know that there exist a reactions in $\mathcal C$ with zero-$V$-reduction. Therefore, there exists a $k \in \{ 1, \dots , r /2\}  $ such that the reaction $R_k = \textbf{R} e_k$ is such that $R_k \in c $ (i.e. $c(k) \neq 0$) for $c \in \mathcal C $ and $\pi_V R_k=0$. 
    As a consequence, the detailed balance property of the kinetic system $(\Omega, \mathcal R, \mathcal K) $ implies that for every $c \in \mathcal C$ we have that
    \begin{align*}
   0&=     \sum_{j=1}^{r/2} c(j) \mathcal E(R_j) =    \sum_{ \{ j: \pi_V R_j =0\} } c(j) \mathcal E(R_j) +   \sum_{ \{ j: \pi_V R_j \neq 0 \} } c(j) \mathcal E(R_j) \\
   &=    \sum_{ \{ j: \pi_V R_j = 0 \} } c(j) \sum_{i \in U } E(i) R_j(i) +   \sum_{ \{ j: \pi_V R_j \neq 0 \} } c(j) \mathcal E(R_j).
    \end{align*}
    As a consequence we have 
    \begin{equation} \label{eq with zeros}
  \sum_{ \{ j: \pi_V R_j = 0 \} } c(j) \sum_{i \in U } E(i) R_j(i) =-   \sum_{ \{ j: \pi_V R_j \neq 0 \} } c(j) \mathcal E(R_j).
    \end{equation}
    Since the reactions in the cycle $\mathcal C_V $ are $1$-$1$ we have that $p[c] \in \mathbb R^{|\{  i \in \{ 1, \dots, r/2  \}: \pi_V R_i \neq 0 \} | } $ and $p[c] (i) = c (i) $ for every $i $ such that $\pi_V R_i \neq 0.$ Recall that $p[c] \in \mathcal C_V $ and recall that by assumption we know that $p[c] \neq 0$.
    As a consequence if we assume that the kinetic system $(V, \mathcal R_V, \mathcal K [n_U]) $ satisfies the detailed balance condition, then we must have that 
    \[ 
     \sum_{ \{ j: \pi_V R_j \neq 0 \} } c(j) \mathcal E(\pi_V R_j)=0.
    \] 
  Notice that this implies that 
  \[
 0=\sum_{ \{ j: \pi_V R_j \neq 0 \} } c(j) \mathcal E(\pi_V R_j)= \sum_{ \{ j: \pi_V R_j \neq 0 \} } c(j) \mathcal E(R_j)  - \sum_{ \{ j: \pi_V R_j \neq 0 \} } c(j)\sum_{i \in U} \log(n_i) R_j(i).
  \]
  Hence we have that
    \[
\sum_{ \{ j: \pi_V R_j \neq 0 \} } c(j) \mathcal E(R_j)  = \sum_{ \{ j: \pi_V R_j \neq 0 \} } c(j)\sum_{i \in U} \log(n_i) R_j(i)
  \]
  This together with equality \eqref{eq with zeros} implies that \begin{align*} 
  \sum_{ \{ j: \pi_V R_j \neq 0 \} } c(j)\sum_{i \in U} \log(n_i) R_j(i) &= \sum_{ \{ j: \pi_V R_j \neq 0 \} } c(j) \mathcal E(R_j) = - \sum_{ \{ j: \pi_V R_j = 0 \} } c(j) \mathcal E(R_j) \\
  &= - \sum_{ \{ j: \pi_V R_j = 0 \} } c(j) \sum_{i \in \Omega } E(i) R_j(i) = - \sum_{ \{ j: \pi_V R_j = 0 \} } c(j) \sum_{i \in U } E(i) R_j(i). 
  \end{align*}
  As a consequence in order to have that the reduced kinetic system satisfies the detailed balance property we need to have 
  \begin{equation} \label{3rd DB}
    \sum_{ \{ j: \pi_V R_j \neq 0 \} } c(j)\sum_{i \in U} \log(n_i) R_j(i)  + \sum_{ \{ j: \pi_V R_j = 0 \} } c(j) \sum_{i \in U } E(i) R_j(i)=0. 
  \end{equation}
  Assume that $n_i$ is not given by \eqref{eq:conc at equilibirum}. Then we can construct a perturbed kinetic system that is such that $(V, \mathcal R_V, \mathcal K_\delta )$ satisfies the detailed balance property but the reduced system does not satisfy \eqref{3rd DB} and hence does not satisfy the detailed balance property.
To this end we can argue exactly as in the proof of Proposition \ref{prop: instability of (DB) different cycles}. 
Indeed, we can define the perturbed rate function $\mathcal K_\delta : \mathcal R \rightarrow \mathbb R_+ $ as $\mathcal K_\delta (R)=\mathcal K (R)$ for every $R \in \mathcal R_s $ and  \[
\mathcal K_\delta (-  R)  =\mathcal K_\delta ( R) e^{\sum_{i \in \Omega }  R(i) (E(i)+ E_\delta (i))} =\mathcal K(-R)  e^{\sum_{i \in \Omega }  R(i)  E_\delta (i)} 
\] 
for every $R \in \mathcal R_s$. Here we have that $E_\delta (i)=0$ for every $i \neq \ell $ and $E_\delta (\ell) <  \frac{\delta}{\max_{R \in \mathcal R_s} |R(\ell)|\max_{ R \in \mathcal R } \mathcal K(R)}$ for $\ell \in U$ such that 
\[ \sum_{ \{ j: \pi_V R_j = 0 \} } c(j)  R_j(\ell) \neq 0.
\]
Notice that such an $\ell \in U $ exists. Indeed assume by contradiction that for every $ i \in U $
\[ 
\sum_{ \{ j: \pi_V R_j = 0 \} } c(j)  R_j(i) = 0. 
\]
This implies that $
\sum_{ \{ j: \pi_V R_j \neq 0 \} } c(j)  R_j(i) = 0$. Hence (upon reordering the reactions), the vector $\tilde{c}\in \mathbb R^{r/2}$ defined as $\tilde c(i)=c(i) $ for every $ i$ such that $ \pi_V R_i \neq 0$ and such that $\tilde c (i)=0$ otherwise is a cycle, i.e. it belongs to $ \mathcal C $ and is such that $p[\tilde c]=c$. This contradicts the assumption of the Proposition. 

As in Proposition \ref{prop: instability of (DB) different cycles} we prove that the rate function $\mathcal K_\delta $ satisfies the assumptions of the Proposition. Moreover, this new rate function does not satisfy \eqref{3rd DB} with respect to the perturbed energy $\overline E= E+ E_\delta$ indeed 
  \begin{align*}
    \sum_{ \{ j: \pi_V R_j \neq 0 \} } c(j)\sum_{i \in U} \log(n_i) R_j(i)  + \sum_{ \{ j: \pi_V R_j = 0 \} } c(j) \sum_{i \in U } (E(i)+ E_\delta(i) ) R_j(i)= E_\delta(\ell) \sum_{ \{ j: \pi_V R_j = 0 \} } c(j)  R_j(\ell) . 
  \end{align*}
  Now recall that $ \sum_{ \{ j: \pi_V R_j = 0 \} } c(j)  R_j(\ell) \neq 0$. 
\end{proof}

We conclude by summarizing the results obtained in this section. 
In Theorem \ref{thm:same cycles implies db} we state some sufficient conditions that guarantee that if the non reduced kinetic system satisfies the detailed balance property, then also the reduced kinetic system satisfies the detailed balance property in a stable manner. 
In Proposition \ref{prop: instability of (DB) no one to one} we have proven that the fact that the cycles of the reduced system contain only reactions that are $1$-$1$ is a necessary condition for the stability of the detailed balance property of the reduced kinetic system. However this property is not sufficient to have that the reduced kinetic system satisfies the detailed balance property in a stable way. This is shown by Example \ref{ex:zero no Db}. 

In Proposition \ref{prop: instability of (DB) one to one and zeros}, we find another necessary condition for the stability of the detailed balance property in the reduced system. This necessary condition is the following. If a cycle $c_V$ of the reduced system can be written as the projection $p[c]$ of a cycle $c$ of the non reduced kinetic system, then $c$ does not contain reactions that project to zero.
Finally in Proposition 
\ref{prop: instability of (DB) different cycles} we prove that the fact that the projection of the cycles of the non reduced system is equal to the set of the cycles of the reduced system is also a necessary condition for the stability of the detailed balance property. 
Summarizing in this section we find three necessary and sufficient conditions conditions to obtain a reduced kinetic system that satisfies the detailed balance property in a stable manner (i.e. stable under small perturbation of the chemical rates or of the frozen concentrations).

\section{Completion of a kinetic system} \label{sec:completion}
As already mentioned in the previous sections, every kinetic system that does not exchange matter with the environment must satisfy the detailed balance property. However, as explained in the previous section, kinetic systems that do not satisfy the detailed balance property can be obtained as a result of the reduction of a kinetic system with detailed balance. 
In this section we characterize the kinetic systems that do not satisfy the detailed balance condition and that are obtained by the reduction of closed kinetic system. In particular, as expected, in Theorem \ref{thm:completion} we prove that every kinetic system that is not closed can be obtained as the reduction of a closed kinetic system.
Notice that in the literature there are many examples of chemical systems that are not closed, see for instance the models of adaptation in \cite{ferrell2016perfect}.

We then consider the case in which some reactions of the open kinetic system cannot be modified, for instance because the details of these reactions are known. 
In this case we prove that it is possible to obtain this open kinetic system as a reduction of a larger closed kinetic system (its completion) if  the reactions that cannot be modified are not part of the cycles of the reduced system. 
We also prove that if all the reactions in the cycles of the kinetic system cannot be modified then the kinetic system does not admit a closed completion. 
However, we stress that we do not prove that an open kinetic system admits a closed completion if and only if all the reactions in its cycles can be modified. Indeed as shown in Example \ref{ex: 1 free reaction is enough}, in some cases, it is enough to be able to modify one single reaction for each cycle to be able to construct a closed completion of the open kinetic system.  

\begin{definition}[Completion of a chemical network and of a kinetic system]
   Let $(\Omega, \mathcal R) $ be a bidirectional chemical network. 
A completion of $(\Omega, \mathcal R) $ is a chemical network $(\Omega_c, \mathcal R_c) $ such that 
\begin{enumerate}
\item $\Omega \subset \Omega_c$; 
\item $\mathcal R= \{ \pi_{\Omega } R : R \in \mathcal R_c \} $. 
\end{enumerate}
A completion of a kinetic system $(\Omega, \mathcal R, \mathcal K)$ is a kinetic system $(\Omega_c, \mathcal R_c, \mathcal K_c) $ that is such that $(\Omega_c, \mathcal R_c)$ is a completion of the chemical network $(\Omega, \mathcal R) $ and there exists a set of concentrations $n_{\Omega_c \setminus \Omega  }:={(n_i)}_{i \in \Omega_c \setminus \Omega  }  $, where $\mathcal K = \mathcal K_c [n_{\Omega_c \setminus \Omega }]$. 
\end{definition}

\subsection{Completion without constraints}
In this section we prove that every open kinetic system $(\Omega, \mathcal R, \mathcal K)$ admits a closed completion.
To this end we first prove that the chemical network $(\Omega , \mathcal R)$ admits a completion $(\Omega_c, \mathcal R_c)$ that is conservative, that is such that every reaction satisfies \eqref{no sources/sinks} and that does not have cycles. Notice that this implies that for every choice of rate function $\mathcal K_c$ we have that $(\Omega_c, \mathcal R_c, \mathcal K_c)$ is closed. Hence every  kinetic system $(\Omega, \mathcal R, \mathcal K)$ admits a closed completion. 

\begin{proposition} \label{prop:remove cycles}
    Assume that $(\Omega, \mathcal R) $ is a bidirectional chemical system. 
Then there exists a completion $(\Omega_c , \mathcal R_c) $ of  $(\Omega, \mathcal R) $ that is conservative, is such that every reaction $R \in \mathcal R_c$ satisfies \eqref{no sources/sinks}, that does not have cycles, i.e. $\mathcal C_c = \{ 0\} $.
Moreover this completion $(\Omega_c , \mathcal R_c) $ is such that every $R \in \mathcal R_c$ is a reaction with non zero-$\Omega $-reduction and every $R \in \mathcal R$ is a $1$-$1$ reduction. 
\end{proposition}
\begin{proof}
    The proof of this proposition is structured as follows. 
As a first step we consider a chemical network that has sources and sinks, hence that does not satisfy \eqref{no sources/sinks}. We explain how to complete this system by adding new substances and to obtain a completion $(\Omega_q, \mathcal R_q) $ whose reactions satisfy \eqref{no sources/sinks}. 
As a second step we assume that $(\Omega_q, \mathcal R_q ) $ is non-conservative and show how to complete it to obtain a completion $(\Omega_e, \mathcal R_e)$ that is conservative. Moreover the way in which the completion is constructed guarantees that no sources and sinks are produced in the completion, hence $(\Omega_e, \mathcal R_e)$ does not have sources and sinks. 
As a third step we show how to complete the chemical network $(\Omega_e, \mathcal R_e)$ to obtain a chemical network $(\Omega_c, \mathcal R_c) $ that has no cycles. Also in this case we construct the completion in such a way that no sources and sinks are produced and in such a way that $(\Omega_c, \mathcal R_c) $ is conservative. 

\textbf{Step 1: Constructing a completed kinetic system that does not have sources and sinks.}
Assume that $(\Omega, \mathcal R) $ is such that the set $Q \subset \mathcal  R $ defined as
\[
Q:= \{ R \in \mathcal R : I(R) = \emptyset \text{ or } F(R)=\emptyset \}
\]
is not empty. Let us define $q:=|Q|$. 
Without loss of generality we consider a reaction $\overline R \in Q$ such that either $I(\overline R)=0$ or $F(\overline R)=0$. 
We define the extended set of substances $\Omega_1 :=\Omega \cup \{N+1, N+2 \} $. 
The set of extended reactions is defined as 
\[
\mathcal R_1:= \left\{ \left( \begin{matrix} R \\
 0 \\
 0
\end{matrix} \right): R \in \mathcal R \text{ s.t. } R \neq \overline R  \text{ and }  -R \neq \overline R  \right\} \cup \left( \begin{matrix} \overline R \\
 \alpha \\
 -\alpha 
\end{matrix} \right) \cup \left( \begin{matrix} -\overline R \\
- \alpha \\
 \alpha 
\end{matrix} \right)
\]
for some $\alpha \neq 0 $. 
Notice that by construction we have that the reaction $\overline R_q := (\overline R , \alpha , - \alpha )^T$ is such that $\pi_\Omega \overline R_q = \overline R $ and is such that $I(\overline R_q) \neq 0 $ and $F(\overline R_q) \neq 0$.
We can iterate the procedure adding two substances to the network for every reaction $R \in \mathcal R$ that is such that either $I(R)=\emptyset $ or $F(R)=\emptyset $.
We obtain a chemical network $(\Omega_q, \mathcal R_q) $ that does not contain sources and sinks.
Notice that we have that $\Omega_q= \Omega \cup \{ N+1, N+2q\} $.  
We stress that for every $R \in \mathcal R_q $ we have that $\pi_{\Omega} R \neq 0 $. Moreover if $R_1, R_2\in \mathcal R_q $ are such that $R_1 \neq R_2 $, then $\pi_{\Omega } R_1 \neq \pi_{ \Omega } R_2 $.

\textbf{Step 2: Constructing a conservative network.}
Assume that  $(\Omega_q, \mathcal R_q) $ is non-conservative. We denote with $\mathcal M_q $ the set of the conservation laws of $(\Omega_q, \mathcal R_q) $.  We now explain how to complete it in order to obtain a conservative system.
Since $(\Omega_q , \mathcal R_q )$ is non-conservative, the set $B \subset \Omega_q$ defined as
 \[
 B:= \left\{ u \in \Omega : m (u)=0 \text{ for every } m \in \mathcal M_q \right\}
 \] 
 is such that $B \neq \emptyset$. Let $ b = |B|$. 
 
Upon reordering of the elements in $\Omega_q$  we have that $B=\{ N+2q-b +1 , \dots,  N +2q\}$.
We define $A:= \Omega_q \setminus B$. 
Notice that every
$R \in \mathcal R_q $ can be written as $ R=\left(\begin{matrix} \pi_A R \\ \pi_B R
 \end{matrix}\right) $. 
 We define the set of extended substances as $\Omega_e = \Omega_q \cup \{ N+2q+1, \dots, N+b+2q \}  $ and we construct the set of the extended reactions as follows 
 \[
\mathcal R_e := \left\{  R_e = \left(\begin{matrix} \pi_A R \\ \pi_B R  \\ - \pi_B R
 \end{matrix}\right) \in \mathbb R^{N+2q +b}: R \in \mathcal R \right\}.
 \]
 Notice that for every $R \in \mathcal R_e $ we have that $\pi_{\Omega_q} R \neq 0 $. Moreover if $R_1, R_2\in \mathcal R_e $ are such that $R_1 \neq R_2 $, then $\pi_{ \Omega_q } R_1 \neq \pi_{ \Omega_q } R_2 $.
 Moreover, each vector $m \in\mathcal M_q $ can be written as 
$ m= \left(\begin{matrix} \pi_A m \\ \textbf{0}_b
 \end{matrix}\right)$. 
Let us define the set of vectors
\[
 M_e := \left\{ m_e =\left( \begin{matrix} \pi_A m  \\ \textbf{1}_{b}\\ \textbf{1}_{b}
 \end{matrix}\right) \in \mathbb R^{N +2q+ b } :m \in \mathcal M_q  \right\}.
 \] 
Now notice that by definition $ M_e \subset \operatorname{span} \{ R_e: R_e \in \mathcal R_e \}^{\perp}$. As a consequence, the chemical network $(\Omega_e, \mathcal R_e ) $, where $\Omega_e= \Omega \cup \{N+2q+1, \dots, N+2q+b \} $ is conservative. Moreover, notice that by construction $(\Omega_e, \mathcal R_e ) $ satisfies \eqref{no sources/sinks}.

\textbf{Step 3: Constructing a completed kinetic system that does not have cycles.}
Assume that the chemical network $(\Omega, \mathcal R ) $ contains some cycles. Then the kinetic system $(\Omega_e , \mathcal R_e ) $ obtained in Step 1 and Step 2 is conservative, it satisfies \eqref{no sources/sinks}, but it might have cycles. 
Consider a basis $\{ c_1, \dots, c_d  \}$ of the space of the cycles $ \mathcal C_e $ of the chemical network $(\Omega_e , \mathcal R_e ) $ and consider a vector $c_i$ of the basis. By definition we have that $\textbf{R}_e c_i =\textbf{0} $. This implies that there exist $k, j $ with $k \neq j $ such that $c_i(k) , c_i(j) \neq 0 $ and a row $\zeta^T \in \mathbb R^{r/2}$ of the matrix $\textbf{R}_e $ such that $\zeta (j) \neq 0 $ and $\zeta(k) \neq  0 $. Here we used the notation $r=|\mathcal R_e|$.

We consider the vector $W_1 \in \mathbb R^{r/2} $ such that $W_1(k) = \zeta (k) $ and $W_1(\ell) =0$  for every $ \ell \neq k $. On the other hand we define $W_2 $ as $\zeta - W_1 $. Notice that by construction $W_1^T c_i \neq 0 $ and $W_2^T c_i \neq 0$.
Therefore the matrix 
\[
\textbf{R}_c =\left(  \begin{matrix}
  \textbf{R}_e \\
 W_1^T \\
   W_2^T \\
\end{matrix} \right) 
\]
is such that $\textbf{R}_c c_i \neq 0$. The set of reactions $\mathcal R_e $ induced by the matrix $\textbf{R}_e$ is such that for every $R \in \mathcal R_e $ we have that $\pi_{\Omega_e} R \neq 0 $. Moreover if $R_1, R_2\in \mathcal R_c $ are such that $R_1 \neq R_2 $, then $\pi_{\Omega_e  } R_1 \neq \pi_{\Omega_e} R_2 $.
Therefore the dimension of the space of the cycles of the chemical network $(\Omega_e \cup \{ N+2q+b,   N+2q+b+2\}, \mathcal R_c) $ is $1 $ less than the dimension of the space of the cycles of $(\Omega_e, \mathcal R_e)$. 
Moreover notice that the network that we construct with this procedure is still conservative. Indeed, without loss of generality let us assume that the vector $\zeta $ is the $ i$-th row of $\textbf{R}_e$, hence $\zeta = e_i^T \textbf{R}_e $. Since $(\Omega_e, \mathcal R_e) $ is conservative we have that there exists $m_e \in \mathcal M_e $ such that $m_e(i ) \neq 0 $.  
Therefore, by construction  the vector $(m_c , \frac{1}{4}m_e(i) , \frac{1}{4} m_e(i) )\in \mathbb R^{N+ 2 q+b+2 } $, where $m_c (j)=m_e(j) $ for every $j \neq i $ and $m_c(i) =\frac{1}{2} m_e(i)$  is a conservation law for $(\Omega_c, \mathcal{R}_c )$.  Finally notice that the kinetic system satisfies \eqref{no sources/sinks} by construction. 
Iterating this process we can remove all the cycles and the desired conclusion follows. 
\end{proof}

\begin{theorem} \label{thm:completion}
Assume that $(\Omega, \mathcal R, \mathcal K) $ is a bidirectional kinetic system. 
Then $(\Omega, \mathcal R, \mathcal K) $ admits a closed completion $(\Omega_c, \mathcal R_c, \mathcal K_c) $ that is such that every $R \in \mathcal R_c$ is a reaction with non zero-$\Omega $-reduction and every $R \in \mathcal R$ is a $1$-$1$ reduction.
\end{theorem}
\begin{proof}
Proposition \ref{prop:remove cycles} implies that the chemical network $(\Omega, \mathcal R) $ admits a completion $(\Omega_c, \mathcal R_c) $ that is conservative, that is such that every $ R \in \mathcal R_c $ satisfy \eqref{no sources/sinks} and that does not have cycles, hence $\mathcal C_c \neq \{ 0\} $.  
As a consequence for every rate function $\mathcal K_c $ we have that $(\Omega_c, \mathcal R_c, \mathcal K_c) $ is closed, indeed since it has no cycles the circuit condition in Lemma \ref{lem:Db equivalence} holds independently on the reaction rates. In particular the kinetic system $(\Omega_c,\mathcal R_c, \mathcal K) $ is closed and is such that $\mathcal K[n_U] = \mathcal K $ for $n_{\Omega_c \setminus \Omega } = \textbf{1} $. Hence $(\Omega_c,\mathcal R_c, \mathcal K) $ is a closed completion of $(\Omega,\mathcal R, \mathcal K) $. 
\end{proof}
\begin{remark}
    Notice that to prove Theorem \ref{thm:completion} we prove that $(\Omega_c,\mathcal R_c, \mathcal K) $ is a closed completion of $(\Omega,\mathcal R, \mathcal K) $. However we also have the same statement for every rate function $\mathcal K_c $ defined as 
    \[
    \mathcal K_c(R) = \mathcal K(\pi_\Omega R) \prod_{i\in I(R) \cap \Omega_c \setminus \Omega } n_i^{R(i)}, \quad  \forall R\in \mathcal R_c
    \]
    for any vector $n_{\Omega_c \setminus \Omega}$. 
\end{remark}

\begin{remark}
   In Proposition \ref{prop:remove cycles} we construct a completion for the chemical network $(\Omega , \mathcal R)$. The completion that we construct does not contain cycles. We then use this completion to prove Theorem \ref{thm:completion}. In particular we use the fact that a kinetic system without cycles satisfies the detailed balance property. However notice that the proof could be optimized. Indeed we could obtain a completion that is closed adding less substances and modifying less reactions. For example, we could construct a completion were we remove only the cycles of the kinetic system $(\Omega , \mathcal R, \mathcal K ) $ on which the circuit condition in Lemma \ref{lem:Db equivalence} does not hold. 
\end{remark}

\subsection{Completion with constraints}
We now consider the case of a kinetic system in which some reactions cannot be modified and study under which conditions it is still possible to obtain a completed kinetic system that is closed.
Here when we say that we cannot modify a reaction $R$ we mean that the only possible completion $\overline R$ of the reaction $R$  has the form $ \left( \begin{matrix}
R \\
\textbf{0}
\end{matrix}  \right) $. In other words the completion of $R $ involves exactly the same substances as $R$, i.e $ D(R )= D(\overline R) $. 

\begin{definition}[Admissible completions]
    Assume that $(\Omega, \mathcal R) $ is a chemical network. 
    Let $\mathcal R_a \subset \mathcal R$. We say that a completion $(\Omega_c, \mathcal R_c) $ of $(\Omega, \mathcal R) $ is $\mathcal R_a$-admissible if 
    \[
    \pi_{\Omega_c \setminus \Omega } R =0 \quad  \forall R \in \mathcal R_c \text{ s.t. } \pi_{ \Omega } R \in \mathcal R_a
    \]
    We say that $\mathcal R_a$ is the set of constrained reactions. 
\end{definition}

In the following proposition we formulate the assumptions on the set of constrained reactions $\mathcal R_a$ that guarantee that a kinetic system admits a closed completion. 
These conditions are that the constrained reactions do not belong to the cycles of the kinetic system that we want to complete, that the admissible reactions are not sources or sinks and that all the substances that appear in the constrained reactions appear in a conservation law. 
\begin{proposition}
    Assume that $(\Omega, \mathcal R, \mathcal K) $ is a bidirectional kinetic system. Let $\mathcal R_a \subset \mathcal R$ be such that $\mathcal R_a \cap \mathcal C = \emptyset$.  Assume that for every $R \in \mathcal R_a $ we have that $\eqref{no sources/sinks} $ holds. Moreover, assume that for every $R \in \mathcal R_a$ we have that $\pi_A R =0 $ where 
    \[
    A:= \{ i \in \Omega : m(i)=0 \text{ for every } m \in \mathcal M\}. 
    \]
Then, there exists a $\mathcal R_a$-admissible closed completion of $(\Omega, \mathcal R, \mathcal K) $. 
\end{proposition}
\begin{proof}
    First of all notice that since by assumption the set of reactions $\mathcal R_a $ satisfy \eqref{no sources/sinks}, then we can repeat step $1$ in the proof of Theorem \ref{thm:completion} to obtain an admissible completion $ (\Omega_q, \mathcal R_q )$ whose reactions satisfy \eqref{no sources/sinks}.  Moreover, since we assume that for every $R \in \mathcal R_a$ we have that $\pi_A R =0 $ we can apply step 2 in the proof of Theorem \ref{thm:completion} in order to obtain an admissible completion that satisfies \eqref{no sources/sinks} and that is conservative.
    
   Finally since the reactions in the set $\mathcal R_a $ are not in the cycles of the kinetic system $(\Omega, \mathcal R, \mathcal K ) $ we can argue as in step 3 in the proof of Theorem \ref{thm:completion} to obtain the desired conclusion.
   Indeed, by construction the vector $\zeta^T \in  \mathbb R^{|\mathcal R|}$ in that proof is such that $\zeta(i)=0$ for every $ i $ such that $R_i \in \mathcal R_a$. The same holds for the vectors $W_1 $ and $W_2 $ constructed in that proof.   
\end{proof}

In the following proposition we prove that if $(\Omega , \mathcal R, \mathcal K ) $ is a kinetic system that does not satisfy the detailed balance property and one of the cycles in the kinetic system contains only reactions that are constrained, then the kinetic system does not admit a closed admissible completion. 
\begin{proposition}
    Assume that the bidirectional kinetic system $(\Omega, \mathcal R, \mathcal K ) $ is such that there exists a cycle $c \in \mathcal C $ with 
    \begin{equation}\label{eq:failed db in cycles}
    \sum_{j=1}^{|\mathcal R|/2} \mathcal E (R_j) c(j) \neq 0.  
    \end{equation}
    Assume moreover that every $R \in c $ is such that $R \in  \mathcal R_a $ where $\mathcal R_a $ is the set of constrained reactions. 
Then the kinetic system $(\Omega, \mathcal R, \mathcal K ) $ does not admit a closed completion. 
\end{proposition}
\begin{proof}
    Assume that the chemical network $(\Omega_c, \mathcal R_c) $ is an admissible completion of $(\Omega, \mathcal K )$. We now prove that $c \in \mathcal C_c $. 
    Indeed since the completion is admissible we have that given a reaction $R \in c $ it holds that there exists a unique completion $\overline R$ of $R $ in $\mathcal R_c$ and this completion is such that $\pi_{\Omega_c \setminus \Omega } \overline R =0$. 
    As a consequence we have that the fact that $\sum_{j=1}^{|\mathcal R|} c(j) R_j=0$ implies that $\sum_{j=1 }^{|\mathcal R_c|}  c(j) \overline R_j=0$. Therefore $c \in \mathcal C_c$. As a consequence for every choice of concentrations $n_{\Omega_c \setminus \Omega } $, the corresponding rate function $\mathcal K_c$ defined as 
    \[
    \mathcal K_c(R)=\mathcal K(R) \prod_{i \in \Omega_c \setminus \Omega } n_i^{R(i)} \  \forall R \in \mathcal R_a 
     \] 
    is such that
    \begin{align*}
     \sum_{j=1}^{|\mathcal R_c|/2} \mathcal E (R_j) c(j)=    \sum_{j=1}^{|\mathcal R_c|/2} \mathcal E ( \pi_\Omega R_j) c(j) +    \sum_{j=1}^{|\mathcal R_c|/2} c(j) \sum_{i \in \Omega_c \setminus \Omega } R_j(i) \log(n_i) =   \sum_{j=1}^{|\mathcal R_c|/2} \mathcal E ( \pi_\Omega R_j) c(j) \neq 0.
    \end{align*}
    In the above computation we have that $\mathcal R_c = \{ R_j\}_j $. As a consequence we have that there is no completion of $(\Omega, \mathcal R, \mathcal K) $ satisfying the detailed balance property.  
\end{proof}

Notice that in the above proposition we need to have that all the reactions in a cycle are constrained reactions in order to have that the kinetic system does not have a closed admissible completion. 
If instead only few reactions on the cycles are constrained it might still be possible to obtain a closed completion as shown in the following example. 
\begin{example} \label{ex: 1 free reaction is enough}
    Consider the kinetic system $(\Omega, \mathcal R, \mathcal K ) $ where $\Omega := \{ 1,2,3,4\} $ and where the reactions are 
        \[
    (1)  \leftrightarrows (2) ,\quad  (2)  \leftrightarrows  (3) , \quad  (3)  \leftrightarrows (4), \quad  (4)  \leftrightarrows (1) . 
    \]
    i.e. 
    \[
    \textbf{R} =\left(  \begin{matrix}
   -1 & 0 & 0 & 1 \\     
    1 & -1 & 0 & 0 \\ 
    0 & 1 & -1 & 0 \\
    0 & 0 & 1 & -1 \\ 
    \end{matrix}\right) 
    \]
    Assume that the rate function $\mathcal K $ is such that the kinetic system does not have the detailed balance property. We assume moreover that the reaction $(2)  \leftrightarrows  (3)$ , i.e. the reaction $(0,-1,1,0)^T$ is constrained. 
    We consider the following admissible completion $\Omega_c := \{ 1,2,3,4, 5 , 6\}$ and the following matrix of the reactions 
   \[
    \textbf{R}_c =\left(  \begin{matrix}
   -1 & 0 & 0 & 1 \\     
    1 & -1 & 0 & 0 \\ 
    0 & 1 & -1 & 0 \\
    0 & 0 & 1 & -1 \\ 
    0 & 0 & 0 & -1 \\ 
    0 & 0 & 0 & 1 \\ 
    \end{matrix}\right) 
    \]  
    Notice that the chemical system $(\Omega_c, \mathcal R_c) $ does not have cycles, is conservative and all the reactions satisfy \eqref{no sources/sinks}, hence, for every choice of rate function $\mathcal K_c $ we have that the kinetic system $(\Omega_c, \mathcal R_c, \mathcal K_c) $ is closed. Therefore the original system $(\Omega, \mathcal R, \mathcal K ) $ has an admissible closed completion.
\end{example}

\section{Kinetic systems with fluxes} \label{sec:kinetic system with fluxes}
Let $(\Omega, \mathcal R, \mathcal K ) $ be a kinetic system. Assume that $U $ and $V$ are as in Definition \ref{def:reduced chemical netowork}. 
We consider now kinetic systems with fluxes of substances in $U \subset \Omega$. 
The concentration of the substances $i \in U $ are fixed at some given values $n_U=(n_i)_{i \in U}$. 
In other words we assume that the concentrations of substances in the network evolve according to the following system of ODEs 
\begin{equation} \label{ODEs with fluxes}
 \frac{d n (t) }{dt } = \sum_{ R \in \mathcal R_s } R J_R(n) + J^{E} (t),  \text{ with } n_0 \in \mathbb R_*^N \text{ s.t. }  \pi_U n(0)=n_U  
\end{equation}
where $ J^E(t) = \sum_{i \in U} J^E_i (t)$ with 
$J^E_i(t) := - e_i \sum_{ R \in \mathcal R_s } R(i) J_R(n) $.
Later we will use the notation  $(\Omega, \mathcal R, \mathcal K ) $ to refer to a kinetic system with fluxes in $U$.

We stress that the kinetic system with fluxes in \eqref{ODEs fluxes} does not have the same stoichiometry and in particular does not have the same conservation laws as the kinetic system $(\Omega, \mathcal R, \mathcal K ) $. 
Indeed, assume that $m \in \mathcal M $. Multiplying equation \eqref{ODEs fluxes} by $m^T$ we obtain that 
\[
 \frac{d m^T n (t) }{dt } = m^TJ^{E}. 
\]
As a consequence, unless $\pi_U m =0 $, we have that $m $ is not a conservation law of the kinetic system with fluxes. 

The rest of the section is organized as follows. In Section \ref{sec:long-time beh system with fluxes} we study the long-time behaviour of kinetic system with fluxes $(\Omega, \mathcal R, \mathcal K , n_U) $ where the kinetic system $(\Omega, \mathcal R, \mathcal K)$ is assumed to satisfy the detailed balance property and where the concentrations $n_U $ are selected to be at equilibrium values. 
In Section \ref{sec:reduced system and system with fluxes} we clarify the relation between the kinetic system with fluxes in $U$, $(\Omega, \mathcal R,\mathcal K , n_U) $ and the reduced kinetic system $(\Omega, \mathcal R,\mathcal K[ n_U ] ) $.

 \subsection{Long-time behaviour of kinetic systems with fluxes}\label{sec:long-time beh system with fluxes}
In this section we prove that when the kinetic system $(\Omega , \mathcal R , \mathcal K ) $ satisfies the detailed balance property and we consider the kinetic system with fluxes in $U $ that take equilibrium values, i.e. $n_U $ is given by \eqref{eq:conc at equilibirum} then the free energy $F$ defined as in \eqref{entropy} is non increasing and the dissipation $\mathcal D= - \partial_t F $ is zero at the steady states of the system of ODEs \eqref{ODEs with fluxes}. 
This allows us to study the long-time behaviour of the solution to  \eqref{ODEs with fluxes} and to obtain the following result. 
\begin{theorem} \label{thm:stability in system with fluxes}
Assume that the kinetic system $(\Omega , \mathcal R, \mathcal K )$ satisfies the detailed balance property. Let $U$ and $V$ be as in Definition \ref{def:reduced chemical netowork}. 
Assume that $U\subset \Omega $ and $n_U$ are such that the assumptions of Proposition \ref{prop: DB when equilibirum conc} hold, hence $n_U $ is given by \eqref{eq:conc at equilibirum}. 
Consider the system of ODEs with fluxes \eqref{ODEs with fluxes} corresponding to $(\Omega , \mathcal R, \mathcal K, n_U )$. 
Then we have that 
\begin{equation} \label{bound fluxes}
 \overline J^E_i:=\int_0^\infty J_i^E (t) < \infty, \quad \text{ for every } i \in U
\end{equation}
where $n$ is the solution to  \eqref{ODEs with fluxes}. 
Moreover we have that for every $n_0\in \mathbb R_*^N $ that is such that $\pi_U n_0 =n_U $ we have that 
$\lim_{t \to \infty } n(t)= e^{-E}$
where $E$ is the unique energy of $(\Omega, \mathcal R, \mathcal K) $ that is such that
\begin{equation} \label{cons laws when fluxes}
m^T e^{- E }- m^T n_0 = m^T \overline J \text{ for every } m \in \mathcal M. 
\end{equation}
Here $\overline J \in \mathbb R_*^{N}$ is the vector $\overline J^E(j)=\overline J_j$ for every $j \in U $ while $\overline J (j)=0$ for every $j \in V$.
\end{theorem}
 \begin{proof}
Let us consider the function $F: \mathbb R_*^N \rightarrow \mathbb R$ defined as in \eqref{entropy}
\[ 
F(n)= \sum_{j \in \Omega } n_j  \left(\log\left( n_j e^{E(j) }\right) - 1 \right)
\]
where $E$ is an energy of $(\Omega, \mathcal R, \mathcal K ) $. 

Notice that using \eqref{ODEs with fluxes} and \eqref{flux DB} and using the fact that $n_U$ is given by \eqref{eq:conc at equilibirum} we deduce that 
\begin{align*}
  \partial_t F (n)&= \sum_{j \in \Omega  }   \partial_t n_j \left(\log\left( n_j e^{E(j) }\right) \right) 
    = \sum_{j \in V }   \partial_t n_j \left(\log\left( n_j e^{E(j) }\right) \right) =
    \sum_{j \in V  }   \sum_{R \in \mathcal R_s } R(j) J_R(n)  \left(\log\left( n_j e^{E(j) }\right) \right) \\
  &=\sum_{R \in \mathcal R_s }  K_R  \prod_{i \in I(R)} n_i^{- R(i) }  \left(1-  \prod_{j\in \Omega  } n_j^{ R(j) } e^{ R(j)E(j) } \right) \log\left( \prod_{j \in V  }  n_j^{R(j)} e^{R(j) E(j) }\right)\\
  &= \sum_{R \in \mathcal R_s }  K_R  \prod_{i \in I(R)} n_i^{- R(i) }  \left(1-  \prod_{j\in V  } n_j^{ R(j) } e^{ R(j)E(j) } \right) \log\left( \prod_{j \in V  }  n_j^{R(j)} e^{R(j) E(j) }\right).
    \end{align*}
  As a consequence we have that $ \partial_t F (n) \leq 0$ for every $n \in \mathbb R_*^N$ and that $ \partial_t F (n) = 0$ if and only if $n(i)=e^{- E(i) }$ for every $ i \in V $. 
  Notice that this in particular implies that there exists a constant $c>0 $ such that 
  \[
 \left| F(n(t)) \right|= \left|\sum_{j \in \Omega } n_j  \left(\log\left( n_j e^{E(j) }\right) - 1 \right) \right| \leq c \text{ for every }  t \geq 0. 
  \] 
  Notice that this implies that there exists a constant $k$ such that $\sum_{j \in \Omega } n_j(t) \leq k  $ for every $t>0$. 
  As a consequence we deduce that for every $j \in \Omega $ it holds that $ n_j (t) < \infty $ for every $t>0$.
On the other hand notice that for every $m \in \mathcal M $ we have that 
\[
m^T  n(t) - m^T  n_0 =m(1) \int_0^t J^E(s) ds
\]
From this we deduce \eqref{bound fluxes}. 
Taking the limit as $ t \rightarrow \infty  $ we deduce that 
  \[
m^T  \lim_{ t \to\infty} n(t) =m^T \overline J  +  m^T  n_0.
\]
Now we prove that given an $n_0 \in \mathbb R_*^N $, there exists an unique energy $E$ of $(\Omega , \mathcal R, \mathcal K ) $ such that $m^T e^{-E}  =m^T \overline J  +  m^T  n_0$. 
Indeed assume that there exists two energies $E_1 $ and $E_2$ that satisfy \eqref{cons laws when fluxes}. Then since $E_1, E_2$ are energies for the same kinetic system $(\Omega, \mathcal R, \mathcal K ) $ we have that $R^T (E_1- E_2) =0$. As a consequence we obtain that there exists a conservation law $\overline m \in \mathcal M $ that is such that $E_1=E_2 + \overline m $. This implies that 
\[
0=\overline m^T (e^{-E_1}-e^{-E_2}) = \overline m^T (e^{-E_2- \overline m}-e^{-E_2}) = \sum_{i \in \Omega } \overline m(i) e^{-E_2(i)} (e^{- \overline m(i)}-1 )  
\]
Notice that $\overline m(i) e^{-E_2(i)}  (e^{- \overline m(i)}-1 ) \geq 0 $ for every $i\in \Omega $ and as a consequence we need to have $\overline m(i) =0$ for every $i \in \Omega $ in order to have $0 = \sum_{i \in \Omega } \overline m(i) e^{-E_2(i)} (e^{- \overline m(i)}-1 ) $. Therefore $E_1=E_2$. 
This implies that for a given initial datum there is a unique energy such that \eqref{cons laws when fluxes} holds. 

As a consequence the restriction of $F$ to the class of solutions satisfying \eqref{cons laws when fluxes} is a radially unbounded Lyapunov functional and $F(n)=0$ if and only if $n=e^{-E} $ for the unique energy of $(\Omega , \mathcal R , \mathcal K) $ that satisfies \eqref{cons laws when fluxes}. We then deduce the desired conclusion. 
\end{proof}
When $|U | =1 $ we have the following result.  
\begin{corollary} 
  Assume that $(\Omega , \mathcal R, \mathcal K )$ satisfies the detailed balance condition.
    Consider the corresponding system with fluxes in $1\in \Omega $, i.e. $(V, \mathcal R_V, \mathcal K , n_1)$ with $ V:=\Omega \setminus \{ 1\} $. Assume that one of the two conditions of Corollary \ref{cor:one frozen conc db} holds. 
    Then we have that the solution $n$ to \eqref{ODEs with fluxes} satisfies \eqref{bound fluxes}. 
Moreover we have that for every $n_0\in \mathbb R^N $ such that $\pi_U n_0 =n_U $ we have that 
$\lim_{t \to \infty } n(t)= e^{-E}$ where $E$ is the unique energy of $(\Omega, \mathcal R, \mathcal K) $ that is such that \eqref{cons laws when fluxes} holds. 
\end{corollary}
\begin{proof}
    The proof of this statement follows by the fact that, as explained in the proof of Corollary \ref{cor:one frozen conc db}, under the assumptions of the corollary for every value of $n_1$ there exists an energy $E$ of the kinetic system $(\Omega , \mathcal R, \mathcal K ) $ that is such that $n_1 = e^{- E(1) }$. The statement then follows by Theorem \ref{thm:stability in system with fluxes}. 
\end{proof}

 \subsection{Reduced kinetic systems and kinetic systems with fluxes} \label{sec:reduced system and system with fluxes}
 Clearly there is a relation between the kinetic system with fluxes $(\Omega , \mathcal R, \mathcal K , n_U)$ and the reduced kinetic system $(\Omega , \mathcal R, \mathcal K [ n_U])$. 
 In this section we clarify how the two systems are related. 
\begin{lemma}[Reduced system and kinetic system with fluxes] \label{lem:reduced kinetics and kinetics with fluxes}
Let $(\Omega , \mathcal R, \mathcal K ) $ be a kinetic system. Let $U $ and $V$ be two sets as in Definition \ref{def:reduced chemical netowork} and let $n_U:= \{ n_i \}_{i \in U } $. 
Then the function $ t \mapsto n(t)  \in \mathbb R_*^{|V|} $ is the solution to the system of ODEs \eqref{ODEs} corresponding to the reduced system $(V , \mathcal R_V, \mathcal K [ n_U ])$ with initial condition $ n_0 \in \mathbb R_*^{|V|} $ if and only if the function $ t \mapsto \overline n (t) \in \mathbb R_*^N $ defined as $\overline n = ( n, n_U ) $ is the solution to \eqref{ODEs with fluxes} with initial condition $\overline n_0 = ( n_0,  n_U )$.
\end{lemma}
\begin{proof}
Notice that by the definition of reduced system, and in particular of $\mathcal K [n_U]$, we have that $n$ is such that
\begin{align*}
  \frac{ d  n }{dt } &= \sum_{R \in (\mathcal R_V)_s } R K_R[n_U] \prod_{i \in I(R)} n_i^{- R(i) }  =  \sum_{R \in (\mathcal R_V)_s} R  \prod_{i \in I(R)} n_i^{- R(i) } \sum_{\overline R \in \pi_V^{-1} (R) } K_{\overline R } \prod_{i \in I(\overline R ) \cap U } n_i^{- \overline R(i) } \\
  &= \sum_{R \in (\mathcal R_V)_s }  \sum_{\overline R \in \pi_V^{-1} (R) } R \prod_{i \in I( \overline R)\cap V }  n_i^{- \overline R(i) }  K_{\overline R } \prod_{i \in I(\overline R ) \cap U  } n_i^{- \overline R(i) } = \sum_{R \in (\mathcal R_V)_s}  \sum_{\overline R \in \pi_V^{-1} (R) } R K_{\overline R } \prod_{i \in I( \overline R) } \overline{n}_i^{- \overline R(i) } \\
  &=  \sum_{\overline R \in \mathcal R_s}  ( \pi_V \overline R)   K_{\overline R } \prod_{i \in I( \overline R) } \overline{n}_i^{- \overline R(i) } =\pi_V \left( \sum_{\overline R \in \mathcal R_s}   \overline R   K_{\overline R } \prod_{i \in I( \overline R) } \overline{n}_i^{- \overline R(i) }\right) =\pi_V \left( \frac{ d \overline n }{dt } \right). 
\end{align*}
Since by the definition of the fluxes $J^E $ in \eqref{ODEs with fluxes} we have that $\pi_U\left( \frac{ d \overline n }{dt }\right)=0 $ the desired conclusion follows. 
\end{proof}

Notice that when the assumptions of Proposition \ref{prop: DB when equilibirum conc} and Corollary \ref{cor:one frozen conc db} are satisfied we have that if the kinetic system $(\Omega, \mathcal R, \mathcal K ) $ satisfies the detailed balance condition, then the reduced kinetic system $(V, \mathcal R_V, \mathcal K [n_U] ) $ satisfies the detailed balance condition. As a consequence of Proposition \ref{prop:global stability db big}, the system of ODEs \eqref{ODEs} corresponding to $(V, \mathcal R_V, \mathcal K [n_U] ) $  has a unique steady state for each compatibility class and the steady state is globally stable. Since, as proven in Lemma \ref{lem:reduced kinetics and kinetics with fluxes} the dynamics of the system of ODEs \eqref{ODEs} corresponding to $(V, \mathcal R_V, \mathcal K [n_U] ) $ is essentially the same as the one of the kinetic system with fluxes $(\Omega , \mathcal R_, \mathcal K, n_U ) $ described by the ODEs \eqref{ODEs with fluxes} when the initial datum in $U$ is given by $n_U $, this provides an alternative way of proving Theorem \ref{thm:stability in system with fluxes}.

We conclude this section with the following proposition where we prove that the change of free energy in a kinetic system with fluxes is due to the contribution of an external source of energy and due to the dissipation of energy in the reduced kinetic system $(U, \mathcal R_V, \mathcal K [n_U])$. 
\begin{proposition}
    Assume that the kinetic system $(\Omega , \mathcal R, \mathcal K )$ satisfies the detailed balance property. Let $U$ and $V$ be as in Definition \ref{def:reduced chemical netowork}. 
Assume that $n_U \in \mathbb R_*^{|U|}$. 
Consider the reduced system $(\Omega , \mathcal R, \mathcal K[ n_U ])$. Assume that every $R \in \mathcal R $ is not a reaction with a zero-$V$-reduction and that every $R \in \mathcal R_V $ is a $1$-$1$ reaction. 
Assume that $n$ is the solution to the system of ODEs with fluxes \eqref{ODEs with fluxes} corresponding to $(\Omega , \mathcal R, \mathcal K, n_U )$. 
Then  
\[
\partial_t F (n)= -  \mathcal D_R(n) +J^{ext}(n). 
\]
where $F$ is given by \eqref{entropy} for an energy $E$ of $(\Omega, \mathcal R, \mathcal K )$ and 
\[
\mathcal D_R(n)= \sum_{R \in \mathcal R_s } K_R[n_U] \prod_{i \in I(R) \cap V} n_i^{- R(i) }  \left(  \frac{  K_{-R}[n_U] }{K_{R}[n_U]} \prod_{i \in V} n_i^{ R(i) }-1\right) \log\left(   \frac{K_{-R}[n_U] }{K_R [n_U] } \prod_{j \in V  } n_j^{R(j)}\right), \quad n \in \mathbb R_+^N
\]
while
\[
J^{ext}(n):= \sum_{j \in U} J^E_j(t) \log(n_j e^{E(j)}). 
\]
\end{proposition}
\begin{proof}
    By the definition of $F  $ we have that 
    \begin{align*}
  \partial_t F &= \sum_{j \in \Omega }   \partial_t n_j \left(\log\left( n_j e^{E(j) }\right) \right) =
    \sum_{j \in V  }   \sum_{R \in \mathcal R_s } J_R(n)  \left(\log\left( n_j^{R(j) } e^{E(j) R(j)  }\right) \right) \\
    &+  \sum_{j \in U  }   \sum_{R \in \mathcal R_s } J_R(n)  \left(\log\left( n_j^{R(j) } e^{E(j) R(j)  }\right) \right) +  \sum_{j \in U  }  J^E_i(t)  \left(\log\left( n_j  e^{E(j)   }\right) \right)  \\
  &=\sum_{R \in \mathcal R_s }  K_R  \prod_{i \in I(R)} n_i^{- R(i) }  \left(1-  \prod_{j\in \Omega  } n_j^{ R(j) } e^{ R(j)E(j) } \right) \log\left( \prod_{j \in \Omega   }  n_j^{R(j)} e^{R(j) E(j) }\right) + J^{ext}(n)
    \end{align*}
    Now notice using the relation between the rates $K_R $ and $K_R[n_U]$ we deduce that
    \begin{align*}
        \mathcal D[n]&= \sum_{R \in \mathcal R_s }  K_R  \prod_{i \in I(R)} n_i^{- R(i) }  \left(1-  \prod_{j\in \Omega  } n_j^{ R(j) } e^{ R(j)E(j) } \right) \log\left( \prod_{j \in \Omega   }  n_j^{R(j)} e^{R(j) E(j) }\right) \\
        &= \sum_{R \in \mathcal R_s }  K_R[n_U] \prod_{i \in I(R) \cap V} n_i^{- R(i) }  \left(1-  \prod_{j\in \Omega  } n_j^{ R(j) } e^{ R(j)E(j) } \right) \log\left( \prod_{j \in \Omega   }  n_j^{R(j)} e^{R(j) E(j) }\right) \\
        &=  \sum_{R \in \mathcal R_s }   \left( K_R[n_U] \prod_{i \in I(R) \cap V} n_i^{- R(i) }-  K_{-R}[n_U] \prod_{i \in F(R) \cap V} n_i^{ R(i) }\right) \log\left( \prod_{j \in \Omega   }  n_j^{R(j)} e^{R(j) E(j) }\right) 
    \end{align*}
    where in the last equality above we have used the fact that the detailed balance property of $(\Omega , \mathcal R, \mathcal K ) $ implies that 
    \[
    K_R[n_U] \prod_{i \in I(R) \cap V} n_i^{- R(i) }  \prod_{j\in \Omega  } n_j^{ R(j) } e^{ R(j)E(j) } = K_{-R}[n_U]  \prod_{i \in F(R) \cap V} n_i^{ R(i) }. 
    \]
    Similarly we have that
    \[
    \prod_{j \in \Omega } n_j^{R(j)} e^{R(j) E(j) } =     \prod_{j \in \Omega } n_j^{R(j)} \frac{K_{-R}}{K_R } =    \prod_{j \in V  } n_j^{R(j)} \frac{K_{-R}[n_U] }{K_R [n_U] }. 
    \]
    As a consequence 
    \begin{align*}
        \mathcal D [n] &= \sum_{R \in \mathcal R_s }   \left( K_R[n_U] \prod_{i \in I(R) \cap V} n_i^{- R(i) }-  K_{-R}[n_U] \prod_{i \in F(R) \cap V} n_i^{ R(i) }\right) \log\left(   \frac{K_{-R}[n_U] }{K_R [n_U] } \prod_{j \in V  } n_j^{R(j)}\right)\\
        &= \sum_{R \in \mathcal R_s } K_R[n_U] \prod_{i \in I(R) \cap V} n_i^{- R(i) }  \left( 1- \frac{  K_{-R}[n_U] }{K_{R}[n_U]} \prod_{i \in V} n_i^{ R(i) }\right) \log\left(   \frac{K_{-R}[n_U] }{K_R [n_U] } \prod_{j \in V  } n_j^{R(j)}\right)=\mathcal D_R[n]. 
    \end{align*}
    \end{proof}
    \begin{remark} 
        Notice that $ F $ is the free energy of the kinetic system with fluxes $(\Omega, \mathcal R, \mathcal K, n_U ) $.
        We obtain that 
        \[
        \partial_t F =- \mathcal D_R + J^{ext}. 
        \]
       Here $  \mathcal D_R $ is the dissipation of free energy and $J^{ext} (n) $ is the external source of free energy due to the in and out fluxes. 
        Notice that $J^{ext}=0$ only if $\pi_U n = e^{- E} $. Hence the source of free energy due to external fluxes is equal to zero only if the concentration of substances in the set $U$ are chosen at equilibrium values. 
        Moreover, notice that $\mathcal D_R[n] =- \partial_t F_V $ where $F_V $ is the free energy of the reduced kinetic system $(\Omega, \mathcal R , \mathcal K [n_U]) $, i.e. 
        \[
        F_V [\pi_V n ] := \sum_{i \in V} n_i \left(  \log (n_i e^{E_V(i) } )-1 \right), 
        \]
        where $E_V $ is an energy of the reduced kinetic system $(\Omega, \mathcal R , \mathcal K [n_U]) $. 
        Notice that if $(\Omega, \mathcal R , \mathcal K [n_U]) $ satisfies the detailed balance property, then $\mathcal D_R[n] \leq 0 $, where $n \in \mathbb R_+^N $.
        While if $(\Omega, \mathcal R , \mathcal K [n_U]) $ does not satisfy the detailed balance property, then there exists a constant $c>0 $ such that $\mathcal D_R[n] \leq -  c$ for  $n \in \mathbb R_+^N $.
        In this sense we can say that considering the completion $(\Omega, \mathcal K , \mathcal R) $ allows to measure the lack of equilibrium of the reduced kinetic system $(V, \mathcal R_V, \mathcal K [n_U])$. 
    \end{remark}
    
    \textbf{Acknowledgements} The authors gratefully acknowledge the support by the Deutsche Forschungsgemeinschaft (DFG) through the collaborative research centre "The mathematics of emerging effects" (CRC 1060, Project-ID 211504053) and Germany's Excellence StrategyEXC2047/1-390685813.
The funders had no role in study design, analysis, decision to publish, or preparation of the manuscript.

    \bibliographystyle{siam}

\bibliography{References}
\end{document}